\documentclass[a4paper, draft]{article}
\usepackage{amsmath}
\usepackage{amsfonts}
\usepackage{amssymb}
\usepackage{amsthm}
\usepackage{tikz}
\usepackage{tikz-cd}
\usetikzlibrary{patterns}
\usepackage{enumerate}

\DeclareMathOperator{\Aut}{Aut}
\newcommand{\bdy}{\ensuremath{\partial}}
\newcommand{\sph}[1]{\ensuremath{\mathbb{S}^{#1}}}

\newcommand{\iso}{\ensuremath{\cong}}
\newcommand{\Z}[1][]{\ensuremath{\mathbb{Z}_{#1}}}

\newcommand{\R}{\ensuremath{\mathbb{R}}}

\newcommand{\E}{\ensuremath{\mathbb{E}}}

\newcommand{\F}{\ensuremath{\mathbb{F}}}
\newcommand{\Nil}{\ensuremath{\mathrm{Nil}}}

\newcommand{\proP}[2][p]{\ensuremath{\widehat{#2}_{(#1)}}}
\newcommand{\nsgp}[1][]{\ensuremath{\triangleleft_{#1}}}
\newcommand{\sbgp}[1][]{\ensuremath{\leq_{#1}}}
\newcommand{\ofg}[1][]{\ensuremath{\pi_1^{\text{orb}}#1}}

\newcommand{\gp}[1]{\ensuremath{{\langle} #1{\rangle}}}
\newtheorem{theorem}{Theorem}[section]

\newtheorem{prop}[theorem]{Proposition}

\newtheorem{lem}[theorem]{Lemma}
\newtheorem{clly}[theorem]{Corollary}

\theoremstyle{definition}
\newtheorem{defn}[theorem]{Definition}
\newtheorem*{cnv}{Conventions}
\theoremstyle{remark}
\newtheorem*{rmk}{Remark}
\newtheorem{example}[theorem]{Example}
\theoremstyle{plain}

\newcounter{introthmcount}
\newenvironment{introthm}[1]{\\[1.9ex] {\bf Theorem\refstepcounter{introthmcount} \label{#1}\Alph{introthmcount}.} \em}{\em \\[2ex]}
\newcommand{\SFS}{Seifert fibre space}
\title{Virtual pro-$p$ properties of 3-manifold groups}
\author{Gareth Wilkes}
\begin{document}
\maketitle
\begin{abstract}
We answer a question of Aschenbrenner and Friedl regarding virtual $p$-efficiency for 3-manifold groups. We then study conjugacy $p$-separability and prove results for Fuchsian groups, Seifert fibre spaces and graph manifolds.
\end{abstract}
\section{Introduction}
Fundamental groups of 3-manifolds are known to have strong residual properties, and well-behaved profinite completions. For instance, the profinite completion of a 3-manifold group determines the geometry of the manifold \cite{WZ14}; and when the manifold is Seifert fibred it determines the isomorphism type of the group \cite{Wilk15}, up to a certain ambiguity found by Hempel \cite{hempel14}. Furthermore orientable 3-manifold groups are conjugacy separable \cite{HWZ12}.

By contrast, the pro-$p$ completion may be very poorly behaved; for example the fundamental group of any knot complement has pro-$p$ completion $\Z[p]$ for any prime $p$. However the pro-$p$ topology is often `virtually' well-behaved, in the sense that a 3-manifold group will have a finite-index subgroup with well-behaved pro-$p$ topology. Aschenbrenner and Friedl \cite{AF13} proved that, for all but finitely many primes $p$, any 3-manifold group is virtually residually $p$. The proviso `all but finitely many primes', arising from the hyperbolic pieces, may be removed in light of the fact that all hyperbolic 3-manifolds are virtually special (by work of Agol, Kahn-Markovic, Wise and others; see \cite{AFW15} for complete referencing) and hence linear over \Z. Koberda \cite{kob13} independently proved that fibred 3-manifolds are virtually residually $p$ for all primes $p$. This property therefore holds for all 3-manifolds except graph manifolds as these are virtually fibred \cite{AGM13, PW12}. Graph manifolds were already known to be virtually residually $p$ for all $p$ by \cite{AF13}.

Aschenbrenner and Friedl \cite{AF13}, as part of their program for obtaining the above result, proved that any graph manifold has a finite-sheeted cover whose JSJ decomposition is `$p$-efficient', meaning that it gives a well-behaved splitting of the pro-$p$ completion. They then asked two questions; firstly whether this property also holds for non-graph manifolds. We exploit virtual fibring and extend the techniques from \cite{kob13} to prove that this does indeed hold: 
\begin{introthm}{introVpE}
Let $M$ be a compact virtually fibred 3-manifold with fibre $\Sigma$ and monodromy $\phi$, where $\Sigma$ is a surface of negative Euler characteristic. Let $p$ be a prime. Then $M$ has a finite-sheeted cover with $p$-efficient JSJ decomposition.
\end{introthm}
Efficiency of the JSJ decomposition plays an important role in establishing conjugacy separability of 3-manifold groups (see \cite{WZ10}, \cite{HWZ12}). The second question asked by Aschenbrenner and Friedl was whether virtual $p$-efficiency has the same application. We apply the techniques from \cite{WZ10} to show that, indeed, graph manifold groups are virtually conjugacy $p$-separable.
\begin{introthm}{introCpS}
Let $M$ be a compact graph manifold. Then $\pi_1 M$ has a finite-index subgroup which is conjugacy $p$-separable.
\end{introthm}
It is not yet known whether hyperbolic 3-manifolds are virtually conjugacy $p$-separable, so we cannot yet extend this to all 3-manifolds. In the course of proving Theorem \ref{introCpS} we prove conjugacy $p$-separability for Fuchsian groups and most Seifert fibre space groups.
\begin{introthm}{introFuchsSFS}
Let $G$ be the fundamental group of a 2-orbifold or of a \SFS{} that is not of geometry \Nil. Then $G$ is conjugacy $p$-separable precisely when $G$ is residually $p$.
\end{introthm}
Conjugacy $p$-separability of surface groups was proved by Paris \cite{Par09}; we give a new proof of this fact.
\begin{cnv}
In this paper, we will use the following conventions.
\begin{itemize}
\item Abstract groups will be assumed finitely presented and will be denoted with Roman letters $G, H, ...$; they will be assumed to have the discrete topology.
\item Profinite groups will be assumed topologically finitely generated and will be denoted with capital Greek letters $\Gamma, \Delta, ...$.
\item The symbols \nsgp[\rm f], \nsgp[p] will denote `normal subgroup of finite index', `normal subgroup of index a power of $p$' respectively; similar symbols will be used for not necessarily normal subgroups.
\item There is a divergence in notation between profinite group theorists, who use $\Z[p]$ to denote the $p$-adic integers, and manifold theorists for whom $\Z[p]$ is usually the cyclic group of order $p$. To avoid any doubt, we follow the former convention and the cyclic group of order $p$ will be consistently denoted $\Z/p$ or $\Z/p\Z$.
\item All manifolds which appear are assumed to be compact and orientable.
\item For us, a {\em graph manifold} will mean a 3-manifold which has a non-trivial JSJ decomposition, all of whose pieces are Seifert fibred. We also insist that the manifold not be a single Seifert fibre space or a Sol manifold. Note that some authors do include these spaces under the name `graph manifold'.
\end{itemize}
\end{cnv}
The author would like to thank Marc Lackenby for carefully reading this paper, and Federico Vigolo for drawing the illustrations. The author was supported by the EPSRC and a Lamb and Flag Scholarship from St John's College, Oxford.
\section{Preliminaries}
Let $G$ be a group. The {\em pro-$p$ topology} on $G$ is the topology whose neighbourhood basis at the identity consists of normal subgroups $N$ of $G$ with $[G:N]$ a power of $p$. Since an intersection of normal subgroups of index a power of $p$ again has index a power of $p$, each normal subgroup of index a power of $p$ contains a characteristic subgroup of index a power of $p$ (that is, a subgroup invariant under all automorphisms of $G$). Thus the characteristic subgroups with index a power of $p$ also form a neighbourhood basis at the identity; we will freely move between these two definitions of the pro-$p$ topology.

A subset $S$ of $G$ is {\em $p$-separable} in $G$ if $S$ is closed in the pro-$p$ topology; equivalently, if for every $g\in G\smallsetminus S$ there is $N\nsgp[p] G$ such that under the quotient map $\phi:G\to G/N$, the image of $S$ does not contain the image of $g$. 

For a subset $S$ of $G$, suppose that for every $g\in G$ not conjugate to any element of $S$, there exists a finite $p$-group $P$ and a surjection $\phi:G\to P$ such that $\phi(g)$ is not conjugate to any element of $\phi(S)$; equivalently suppose that the union of the conjugacy classes of elements in $S$ is $p$-separable. Then we say $S$ is {\em conjugacy $p$-distinguished} in $G$. If $g\in G$, we say $g$ is conjugacy $p$-distinguished in $G$ if $\{g\}$ is conjugacy $p$-distinguished. If all elements of $G$ are conjugacy $p$-distinguished, then $G$ is called {\em conjugacy $p$-separable}.

For $H$ a subgroup of $G$, we say that $G$ {\em induces the full pro-$p$ topology on $H$}, or that {\em $H$ is topologically $p$-embedded in $G$}, if the induced topology on $H$ agrees with its pro-$p$ topology. That is, we require that for any $N\nsgp[p] H$ there is $N'\nsgp[p] G$ such that $N'\cap H\leq N$. Note that if $H$ is a normal subgroup of $G$ with index a power of $p$, then $G$ induces the full pro-$p$ topology on $H$, because any characteristic normal subgroup of $H$ is a normal subgroup of $G$.

We will be needing the language of pro-$p$ groups acting on pro-$p$ trees. A detailed knowledge will not be necessary in this paper; for the present purpose we need only concern ourselves with  some definitions made by analogy with abstract Bass-Serre theory. Let ${\cal G}=(X,G_\bullet)$ be a graph of discrete groups with base graph $X$ and vertex and edge groups $G_v, G_e$ respectively; let $G$ be the fundamental group of this graph of groups, denoted $\pi_1({\cal G})$ or $\pi_1(X,G_\bullet)$. There is a standard tree $T=S(\cal G)$ on which $G$ acts, constructed as follows: the vertex (respectively, edge) set of $T$ consists of cosets of the vertex (respectively, edge) groups $G_x$ in $G$; that is, 
\[V(T) = \coprod_{x\in V(X)} G/G_x,\quad E(T) = \coprod_{e\in E(X)} G/G_e \]
with the obvious incidence maps given by inclusions $gG_e\subseteq gG_x$ when $x$ is an endpoint of $e$. Vertex stabilisers for the action of $G$ on $T$ are conjugates of the $G_x$, and the quotient graph $G\backslash T$ is $X$.

Similarly, given a graph of pro-$p$ groups $\proP{\cal G}=(X,\Gamma_\bullet)$ with fundamental group $\Gamma=\Pi_1(\proP{\cal G})=\Pi_1(X,\Gamma_\bullet)$ (defined by the same universal property as in the abstract case, in the category of pro-$p$ groups), there is a standard tree $S(\proP{\cal G})$ with precisely the same formal definition as above. Again the quotient graph is $X$ and vertex stabilisers have the expected forms. 

Given a graph of discrete groups $(X,G_\bullet)$ one may form a graph of pro-$p$ groups $\proP{\cal G}=(X,\proP{G_\bullet})$ by taking the pro-$p$ completion of each group; one may ask what relation $\Gamma=\Pi_1(\proP{\cal G})$ bears to $G=\pi_1({\cal G})$ and what relation the standard trees bear to one another. In general this relationship may be complicated. However there is a set of conditions which ensure that the behaviour is well-controlled.

\begin{defn}
A graph of discrete groups ${\cal G}=(X,G_\bullet)$ is {\em $p$-efficient} if $G=\pi_1(\cal G)$ is residually $p$, each group $G_x$ is closed in the pro-$p$ topology on $G$, and $G$ induces the full pro-$p$ topology on each $G_x$.
\end{defn}
In the case when our graph of groups is $p$-efficient, then $\Gamma=\proP{G}$ and the abstract standard tree $S(\cal G)$ is canonically embedded in $S(\proP{\cal G})$. 

Note that when $G$ is a free product, i.e.\ all edge groups of $\cal G$ are trivial, $G$ certainly induces the full pro-$p$ topology on its factors, and by the argument in Proposition \ref{bdycptssep} these are $p$-separable; so a free product decomposition of a residually $p$ group is always $p$-efficient.
 
The following property plays a role in conjugacy separability results.
\begin{defn}
An action of a (profinite) group on a (profinite) tree $T$ is {\em $k$-acylindrical} if the stabiliser of any path in $T$ of length greater than $k$ is trivial.
\end{defn}
For instance, `0-acylindrical' refers to an action with trivial edge stabilisers, and `1-acylindrical' says that edge stabilisers are malnormal in vertex groups.

This brief sketch is enough to make the paper readable; for a more detailed discussion, see Chapter 9 of \cite{RZ00}, \cite{RZup} or \cite{RZ00p}.

For background about Fuchsian groups, orbifolds, and Seifert fibre spaces the reader is referred to \cite{scott83}, also to Thurston's notes on orbifolds (\cite{thurstonnotes}, Chapter 13). We recall the criteria for a \SFS{} group or Fuchsian group to be residually $p$; for a proof see \cite{Wilk16}, Section 9.
\begin{lem}
Let $O$ be an orientable orbifold with non-positive Euler characteristic and such that each cone point of $O$ has order a power of $p$. Then $O$ has a regular cover of degree a power of $p$ which is a surface. Hence $\ofg{O}$ is residually $p$.
\end{lem}
\begin{prop}
Let $p$ be a prime. Let $M$ be a \SFS{} which is not of geometry $\sph{3}$ or $\sph{2}\times\R$. Then $\pi_1 M$ is residually $p$ if and only if all exceptional fibres of $M$ have order a power of $p$, and $M$ has orientable base orbifold when $p\neq2$. That is, $M$ has residually $p$ fundamental group precisely when its base orbifold $O$ is $\Z/p$-orientable and has residually $p$ fundamental group.
\end{prop}
\section{Virtual $p$-efficiency}
\begin{prop}\label{curvetopemb}
Let $l$ be an essential simple closed curve on an orientable compact surface $\Sigma$. Then for any $r$, there is some $p$-group quotient of $G=\pi_1 \Sigma$ in which the image of $l$ is $p^r$-torsion. In particular, $G$ induces the full pro-$p$ topology on $L=\pi_1 l$.
\end{prop}
\begin{proof}
We will find, for each integer $r$, a finite $p$-group quotient of $G$ such that the image of $l$ is $p^r$-torsion. If $l$ is non-separating, then it represents a primitive class in $H_1(\Sigma;\Z)$, so a suitable map $G\twoheadrightarrow \Z/p^r$ induces the map $L\twoheadrightarrow\Z/p^r$. 
If $l$ is separating, let $G_1$, $G_2$ be the fundamental groups of the two components $\Sigma_1, \Sigma_2$ of $\Sigma\smallsetminus l$. If $\Sigma_i$ has $l$ as its only boundary component, then $G_i$ is a free group with generators $a_i,b_i \,(1\leq i\leq g)$ in which $l$ is the product of the commutators $[a_i,b_i]$; now define a map from $G_i$ to the `mod-$p^r$ Heisenberg group' 
\[ {\cal H}_3(\Z/p^r) = \left\{\begin{pmatrix} 1 & x & z \\ 0&1&y\\0&0&1 \end{pmatrix} :x,y,z\in\Z/p^r\right\}\]
by mapping \[a_1\mapsto \begin{pmatrix}1 & 1&0 \\ 0 &1&0\\0&0&1\end{pmatrix}, \quad b_1\mapsto \begin{pmatrix}1 & 0&0 \\ 0 &1&1\\0&0&1\end{pmatrix}\] and mapping the other generators to the identity matrix, so that the image of $l$ is the $p^r$-torsion element 
\[\begin{pmatrix}1 & 0&1 \\ 0 &1&0\\0&0&1\end{pmatrix}\]

If $\Sigma_i$ has another boundary component, then $l$ is again a primitive element in the  homology of $\Sigma_i$, so a suitable map to $\Z/p^r$ induces a surjection $L\twoheadrightarrow \Z/p^r$. We may now exhibit the required quotients of $G$, according to the division of cases above: when both $\Sigma_i$ have another boundary component, map $G\to \Z/p^r$ for in this case $l$ is primitive in the homology of $\Sigma$. When $\Sigma_1$ has no boundary component other than $l$, map
\[ G=G_1\ast_L G_2\to {\cal H}_3(\Z/p^r)\ast_{\Z/p^r} \Z/p^r  = {\cal H}_3(\Z/p^r)\]
and when both $\Sigma_i$ have this property, map
\[ G=G_1\ast_L G_2\to {\cal H}_3(\Z/p^r)\ast_{\Z/p^r} {\cal H}_3(\Z/p^r)  \to {\cal H}_3(\Z/p^r)\]
where the final homomorphism identifies the two copies of the Heisenberg group.
\end{proof}
It will follow from the next sequence of propositions that $L$ is also $p$-separable in $G$; for it will be $p$-separable in each $G_i$ by the next proposition and the splitting along $l$ will be $p$-efficient by Propositions \ref{surfacesplittingamalg} and \ref{surfacesplittingHNN}.
\begin{prop}\label{bdycptssep}
Let $\Sigma$ be a compact orientable surface with non-empty boundary that is not a disc. Let $l$ be a boundary component. Then $L=\pi_1 l$ is $p$-separable in $G=\pi_1 \Sigma$.
\end{prop}
\begin{proof}
If $\Sigma$ has only one boundary component and thus has positive genus, then pass to a regular abelian $p$-cover with more than one boundary component. Then $l$ lifts to this cover, and it suffices to prove that $L$ is $p$-separable when $\Sigma$ has more than one boundary component. In this case, $L$ is a free factor of $G$; that is, $G=L\ast F$ for some free group $F$. Let $g\in G\smallsetminus L$, and write $g$ as a reduced word \[g=l^{m_1} f_1 l^{m_2} f_2\ldots f_n\] where the $m_i\in\Z$, $f_i\in F$ are all non-trivial except possibly $f_n$ (when $n>1$) and $m_1$. Then there is a finite $p$-group quotient $F\twoheadrightarrow P$ in which no non-trivial $f_i$ is mapped to the identity; taking $r$ larger than all $m_i$, the image of $g$ under the quotient map \[\phi: G= L\ast F\to \Z/p^r \ast P\] is a reduced word with some letter in $P\smallsetminus \{1\}$; then $\phi(g)\notin \phi(L)$. Since $\Z/p^r \ast P$ is residually $p$, we can pass to a finite $p$-group quotient distinguishing $\phi(g)$ from the (finitely many) elements of $\phi(L)$; this quotient $p$-group separates $g$ from $L$. Hence $L$ is $p$-separable in $G$.      
\end{proof}
\begin{defn}
Let $P$ be a finite $p$-group. A {\em chief series} for $P$ is a sequence 
\[1=P_n \leq P_{n-1}\leq \cdots\leq P_2\leq P_1=P\]
of normal subgroups of $P$ such that each quotient $P_i/P_{i+1}$ is either trivial or isomorphic to $\Z/p$.
\end{defn}
\begin{theorem}[Higman \cite{Hig64}]\label{Higman}
Let $A, B$ be finite $p$-groups with common subgroup $A\cap B=U$. Then $A\ast_U B$ is residually $p$ if and only if there are chief series $\{A_i\},\{B_i\}$ of $A,B$ such that $\{U\cap A_i\}=\{U\cap B_i\}$. In particular, $A\ast_U B$ is residually $p$ when $U$ is cyclic.
\end{theorem}
\begin{prop}\label{surfacesplittingamalg}
Let $\Sigma$ be a compact orientable surface and let $l$ be an essential separating simple closed curve on $\Sigma$. Let $\Sigma_1,\Sigma_2$ be the closures of the two components of $\Sigma\smallsetminus l$. Let $G=\pi_1 \Sigma$, $G_i=\pi_1 \Sigma_i$, and $L=\pi_1 l$. Then $G=G_1\ast_L G_2$ is a $p$-efficient splitting. 
\end{prop}
\begin{proof}
Let $H\nsgp[p] G_1, P_1=G_1/H$, and suppose $LH/H\iso\Z/p^r$. By Proposition \ref{curvetopemb}, there is a $p$-group quotient $G_2\twoheadrightarrow P_2$ such that the image of $L$ is again isomorphic to $\Z/p^r$. The quotient $P_1\ast_{\Z/p^r} P_2$ thus obtained is residually $p$, so admits a $p$-group quotient $Q$ distinguishing all the (finitely many) elements of $P_1$; so the kernel of the composite map 
\[G_1 \to G = G_1\ast_L G_2\to P_1\ast_{\Z/p^r} P_2 \to Q \]
is $H$; so $G$ induces the full pro-$p$ topology on $G_1$. 

Now if $g\in G\smallsetminus G_1$, write \[g=a_1 b_1\cdots a_n b_n\] where all $a_i\in G_1$, $b_i\in G_2$ are not in $L$ (except possibly $b_n =1$ if $n>1$, or possibly $a_1\in L$); that is, write $g$ as a reduced word in the amalgamated free product. Note that $b_1\neq 1$. By Proposition \ref{bdycptssep} we may find $H_i\nsgp[p] G_i$ such that the image of every non-trivial $b_j$ in $P_2=G_2/H_2$ does not lie in the image of $L$ (and similarly for $P_1=G_1/H_1$). Suppose that the image of $l$ in $P_i$ is $p^{r_i}$-torsion, and take $r=\max\{r_1,r_2\}$. By Proposition \ref{curvetopemb} we may find $K_i\nsgp[p] G_i$ such that $K_i\cap L=p^r L$; then replace $H_i$ by the deeper subgroup $H_i\cap K_i$. In this way we ensure that the image of $L$ is $\Z/p^r$ in both $P_1$ and $P_2$, and we may form the amalgamated free product $P_1 \ast_{\Z/p^r} P_2$. By construction the image $\phi(g)$ of $g$ under the quotient $\phi:G=G_1\ast_L G_2\to P_1 \ast_{\Z/p^r} P_2$ is a reduced word with a letter in $P_2$, hence does not lie in $P_1$. Since $P_1 \ast_{\Z/p^r} P_2$ is residually $p$ by Theorem \ref{Higman}, we may find a $p$-group quotient $P_1 \ast_{\Z/p^r} P_2\twoheadrightarrow Q$ distinguishing $\phi(g)$ from $P_1$; hence $G\to Q$ distinguishes $g$ from $G_1$ and so $G_1$ is $p$-separable in $G$.
\end{proof}
\begin{theorem}[Chatzidakis \cite{chat94}]\label{Chatz}
Let $P$ be a finite $p$-group, $A,B$ subgroups of $P$, and $f:A\to B$ an isomorphism. Suppose that $P$ has a chief series $\{P_i\}$ such that $f(A\cap P_i)=B\cap P_i$ for all $i$ and the induced map \[f_i : AP_i\cap P_{i-1}/P_i \to  BP_i\cap P_{i-1}/P_i\] is the identity for all $i$. Then $P$ embeds in a finite $p$-group $T$ in which $f$ is induced by conjugation. Hence the HNN extension $P\ast_A$ is residually $p$.
\end{theorem}
\begin{prop}\label{surfacesplittingHNN}
Let $\Sigma$ be a compact orientable surface and let $l$ be a non-separating simple closed curve on $\Sigma$. Choose a regular neighbourhood $l\times [-1,1]$ of $l$ in $\Sigma$ and let $\Sigma_1=\Sigma\smallsetminus (l\times(-1,1))$. Let $G=\pi_1 \Sigma$, $H=\pi_1 \Sigma_1$, and $A=\pi_1 (l\times\{-1\})$, $B=\pi_1 (l\times\{+1\})$. Let $f:A\to B$ be the natural isomorphism. Then the HNN extension $G=H\ast_A$ is a $p$-efficient splitting.
\end{prop}
\begin{proof}
First we will prove that $G$ induces the full pro-$p$ topology on $H$. Let $P=H / \gamma^{(p)}_n(H)$ be one of the lower central $p$-quotients of $H$, and $\phi:H\to P$ the quotient map. Let $a, b$ denote the generators of $\phi(A),\phi(B)$. Note that since commutator subgroups and terms of the lower central $p$-series are verbal subgroups, there is a commuting diagram:
\[\begin{tikzcd} 
{} & H \ar{ld} \ar{dr} & {} \\
H/\gamma^{(p)}_n(H)=P \ar{d} & {}& H/[H,H] = H_{\rm ab}\ar{d}\\
P/[P,P] \ar{rr}{\iso} & {} & H_{ab}/\gamma^{(p)}_n(H_{\rm ab})
\end{tikzcd}\] so that $P_{\rm ab} =H_{\rm ab}\otimes \Z/p^{n-1}$. The image of $a$ in $P$ thus has order at least $p^{n-1}$; by definition of the lower central series any element of $P$ has order at most $p^{n-1}$. Hence the image of $A$ in $P$ injects into $P_{\rm ab}$. Since $P$ is a characteristic quotient of $H$ and there is an automorphism of $H$ taking $a$ to $b$, the order of $b$ will also be $p^{n-1}$. Furthermore and $a,b$ are mapped to the same element of $P_{\rm ab}$. Now construct a chief series $(P_i)$ for $P$ whose first $n$ terms are the preimages of the terms of a chief series for $P_{\rm ab}$ which intersects to a chief series on the subgroup of $P_{\rm ab}$ generated by the image of $a$. Then for $i\geq n$, we have $\phi(A)\cap P_i = \phi(B)\cap P_i=1$ and for $i<n$ the conditions of Theorem \ref{Chatz} hold by construction. Hence $P\ast_{\phi(A)}$ is residually $p$, and we may take a $p$-group quotient $P\ast_{\phi(A)}\to Q$ in which no element of $P$ is killed; then the kernel of the composite map
\[ H\to G=H\ast_A\to P\ast_{\phi(A)}\to Q\]
is $\gamma^{(p)}_n(H)$ as required. 

To show that $H$ is $p$-separable in $G$, proceed as in the proof of Proposition \ref{surfacesplittingamalg}; that is, write $g\in G\smallsetminus H$ as a reduced word in the sense of HNN extensions, and take a sufficiently deep lower central $p$-quotient $P=H / \gamma^{(p)}_n(H)$ so that the image of $g$ in $P\ast_{\phi(A)}$ is again a reduced word, hence not in $P$. As shown above, $P\ast_{\phi(A)}$ is residually $p$, so admits a $p$-group quotient $Q$ distinguishing the image of $g$ from the image of $P$; this quotient $Q$ of $G$ exhibits that $H$ is $p$-separable in $G$.
\end{proof}
Propositions \ref{surfacesplittingamalg} and \ref{surfacesplittingHNN} together give the following more general result:
\begin{prop}\label{surfacesplitting}
Let $\Sigma$ be a compact orientable surface and let $l_1,\ldots,l_n$ be a collection of pairwise disjoint, non-isotopic, essential simple closed curves in $\Sigma$. Then the splitting of $\Sigma$ along the $l_i$ gives a $p$-efficient graph of groups decomposition of $\pi_1\Sigma$.
\end{prop}
The following proposition is an easy consequence of \cite{DdSMS03}, Propositions 0.8 and 0.10.
\begin{lem}\label{UnipAction}
Let $P$ be a finite $p$-group and suppose $\psi\in \Aut(P)$ acts unipotently on the $\F_p$-vector space $H_1(P;\Z/p)$. Then $\psi$ has $p$-power order.
\end{lem}
Lemma \ref{UnipAction} was used in \cite{kob13} to prove that certain semidirect products are residually $p$. The reader is warned that in the arXiv version \cite{kob09} of that paper, Lemma \ref{UnipAction} is stated in the context of finite nilpotent groups, where it is false.

Lemma \ref{UnipAction} allows us to give a complete characterisation of the pro-$p$ topology on certain semidirect products. First we fix some notation. Let $G,C$ be finitely generated groups, and let $\Phi:C\to\Aut(G)$ be a homomorphism. Denote the automorphism $\Phi(c)$ by $\Phi_c$ and define the semidirect product $G\rtimes C$ to be the set $G\times C$ equipped with group operation 
\[ (g_1,c_1)\star (g_2,c_2) = (g_1\Phi_{c_1}(g_2), c_1 c_2)\]
Identify $G$ with $\{(g,1):g\in G\}$ and $C$ with $\{(1,c):c\in C\}$. There is a function (not a homomorphism of course) $u$ from the semidirect product $G\rtimes C$ to the direct product $G\times C$ `forgetting the map $\Phi$', which is the identity on the underlying sets of the two groups. Note that if $N$ is a characteristic normal subgroup of $G$ and $D$ is a subgroup of $C$, then $N\rtimes D$ is a subgroup of $G\rtimes C$ and $u(N\rtimes D)=N\times D$. 
\begin{prop}\label{semidirecttopology}
Let $G,C$ be finitely generated groups, and let $\Phi:C\to\Aut(G)$ be a homomorphism. Suppose that each automorphism $\Phi_c$ acts unipotently on $H_1(G;\F_p)$. Then the forgetful function $u:G\rtimes C\to G\times C$ is a homeomorphism, where both groups are given their pro-$p$ topology.
\end{prop}
\begin{proof}
We first claim that it suffices to prove the following two statements:
\begin{enumerate}[(i)]
\item For each $U\subseteq G\rtimes C$ a basic open neighbourhood of 1, there exists $V\subseteq G\times C$ open such that $1\in V\subseteq u(U)$.
\item For each $U\subseteq G\times C$ a basic open neighbourhood of 1, there exists $V\subseteq G\rtimes C$ open such that $1\in V\subseteq u^{-1}(U)$.
\end{enumerate}
That is, the neighbourhood bases at 1 match up. For left-multiplication by $(g,1)$ and right-multiplication by $(1,c)$ are continuous both as maps on $G\times C$ and on $G\rtimes C$, and commute with $u$. Thus if $U\subseteq G\rtimes C$ is a basic open neighbourhood of $(g,c)$, then finding $V\subseteq G\times C$ such that $1\in V\subseteq (g^{-1},1)u(U)(1,c^{-1})$ gives a ($G\times C$)-open set $(g,1)V(1,c)$ exhibiting that $u(U)$ is a $G\times C$-neighbourhood of $(g,c)$. Hence (i) implies that $u$ is an open map; similarly (ii) implies that $u$ is continuous. 

Let us prove (i). If $U\nsgp[p] G\rtimes C$ is a basic open neighbourhood of 1, then $U\cap G\nsgp[p] G$ and $U\cap C\nsgp[p] C$, so $V= (U\cap G)\times (U\cap C)$ is a normal $p$-power index subgroup of $G\times C$, so is $(G\times C)$-open. Also $1\in V\subseteq u(U)$ since if $(g,1), (1,c)\in U$ then $(g,c)=u((g,1)\star (1,c))\in U$. So (i) holds and $u$ is an open mapping.

The more difficult statement is (ii). Let $U=N\times D$ be a basic open neighbourhood of 1 in $G\times C$, where $N\nsgp[p] G, D\nsgp[p] C$ and $N$ is characteristic in $G$. Then $N\rtimes D=u^{-1}(U)$ is a subgroup of $G\rtimes C$ with index a power of $p$; however it need not be normal. We will find a deeper subgroup that is normal in $G\rtimes C$, still with index a power of $p$. 

Now $H_1(G/N;\F_p)$ is a quotient of $H_1(G;\F_p)$, on which $\Phi_c$ acts unipotently for every $c\in C$; so by Lemma \ref{UnipAction}, the map induced by $\Phi_c$ on $G/N$ has order a power of $p$. Thus every element of the image of $C\to \Aut(G/N)$ has order a power of $p$; so the image is a finite $p$-group. Let $K\nsgp[p] C$ be the kernel of this map. Each element of $D\cap K$ acts trivially on $G/N$, so we have a quotient map
\[ G\rtimes C\to (G/N)\rtimes C\to (G/N)\rtimes (C/D\cap K) \]
whose kernel $V=N\rtimes (D\cap K)$ is thus a normal subgroup of $G\rtimes C$ with index a power of $p$, and $1\in V\subseteq u^{-1}(U)$. Thus (ii) holds as required.
\end{proof}
\begin{rmk}
Note that the pro-$p$ topology on the group $G\times C$ is the product of the pro-$p$ topologies of $G$ and $C$. In particular if both $G,C$ are residually $p$ then so is $G\times C$, hence under the conditions of the above proposition $G\rtimes C$ is also residually $p$.
\end{rmk}
\begin{theorem}
Let $M$ be a compact fibred 3-manifold with fibre $\Sigma$ and monodromy $\phi$, where $\Sigma$ is a surface of negative Euler characteristic. Let $p$ be a prime. Then $M$ has a finite-sheeted cover with $p$-efficient JSJ decomposition.
\end{theorem}
\begin{proof}
Without loss of generality both $M$ and $\Sigma$ are orientable. Then, possibly after performing an isotopy of the monodromy, the JSJ tori of $M$ intersect $\Sigma$ in a collection of disjoint non-isotopic essential simple closed curves $\{l_1,\ldots, l_n\}$ which are permuted by the monodromy $\phi$. The $l_i$ divide $\Sigma$ into a number of subsurfaces $\Sigma_1,\ldots \Sigma_m$. The monodromy acts on the set of $\Sigma_j$. Each piece of the JSJ decomposition corresponds to an orbit of this action, and is fibred over any element of that orbit. If $n_j$ is the size of the orbit of $\Sigma_j$ then $\phi^{n_j}$ acts on $\Sigma_j$ either periodically or as a pseudo-Anosov. The monodromy $\phi$ also acts on $H_1(\Sigma;\F_p)$; let $k$ be the order of $\phi$ in
\[ {\rm Sym}(\{\Sigma_1,\ldots,\Sigma_m\})\times {\rm Sym}(\{l_1,\ldots,l_n\})\]
Finally take some multiple $k'$ of $k$ such that $\phi^{k'}$ acts by the identity on each $H_1(\Sigma_j;\F_p)$. Let $\psi=\phi^{k'}$ and let $\tilde M$ be the surface bundle over $\Sigma$ with monodromy $\psi$, an index $k'$ cover of $M$. Then $\psi$ fixes each $\Sigma_j$ and $l_i$, and acts on each $\Sigma_j$ periodically or as a pseudo-Anosov, so that the JSJ tori of $\tilde M$ are precisely the tori $l_i\times \sph{1}$, and the pieces of the JSJ decomposition are the mapping tori $\tilde M_j = \Sigma_j\rtimes_\psi \sph{1}$. We claim that $\tilde M$ has $p$-efficient JSJ decomposition. We must show that each vertex (respectively edge group) $\pi_1(\Sigma_j\rtimes_\psi \sph{1})$ (respectively edge group $\pi_1 (l_i\times\sph{1})$) is $p$-separable and inherits the full pro-$p$ topology from $\pi_1 \tilde M$. We prove this statement for the vertex groups, the proof for edge groups being similar.

Choose a basepoint $x\in \Sigma_j$ and a loop $\gamma$ lying in $\tilde M_j$ transverse to the fibres and passing through $x$. The homotopy class of $\gamma$ gives a splitting of the quotient map to \Z[] coming from the fibration, hence gives an identification of $\pi_1 \tilde M$ with $\pi_1 \Sigma\rtimes_{\psi} \Z$ in which the vertex group $\pi_1 \tilde M_j$ is embedded as $\pi_1 \Sigma_j \rtimes_\psi \Z$. The forgetful function $u:\pi_1 \Sigma\rtimes_{\psi} \Z\to \pi_1 \Sigma\times \Z$ now sends $\pi_1 \Sigma_j \rtimes_\psi \Z$ to $\pi_1 \Sigma_j \times \Z$. The action of $\psi$ on each $H_1(\Sigma_i;\F_p)$ is unipotent by construction, hence also is the action on $H_1(\Sigma;\F_p)$. Hence by Proposition \ref{semidirecttopology}, $u$ is a homeomorphism of pairs
\[(\pi_1 \tilde M,\pi_1 \tilde M_j)=(\pi_1 \Sigma \rtimes_\psi \Z,\pi_1 \Sigma_j \rtimes_\psi \Z)\to (\pi_1 \Sigma \times \Z,\pi_1 \Sigma_j \times \Z)\]
By Proposition \ref{surfacesplitting}, $\pi_1 \Sigma_j$ is $p$-separable in $\pi_1 \Sigma$ and inherits its full pro-$p$ topology. The same is thus true of $\pi_1 \Sigma_j\times\Z$ in the product topology; the homeomorphism $u$ now yields the result.  
\end{proof}

\section{Conjugacy $p$-separability}
In \cite{WZ10} Wilton and Zalesskii proved a combination theorem for conjugacy separability. The proof of this uses the theory of profinite groups acting on profinite trees. The parallel theory for pro-$p$ groups yields the following theorem:
\begin{theorem}\label{CpScombination}
Let ${\cal G}=(X, G_\bullet)$ be a graph of groups with conjugacy $p$-separable vertex groups $G_v$. Let $G=\pi_1({\cal G})$ and suppose that the graph of groups $\cal G$ is $p$-efficient and that the action of $\proP{G}$ on the standard tree of $\proP{\cal G}=(X, \proP{G_\bullet})$ is 2-acylindrical. Suppose that the following conditions hold for any vertex $v$ of $X$ and any incident edges $e,f$ of $v$ in $X$:
\begin{enumerate}
\item for any $g\in G_v$ the double coset $G_e g G_f$ is $p$-separable in $G_v$;
\item the edge group $G_e$ is conjugacy $p$-distinguished in $G_v$;
\item the intersection of the closures of $G_e$ and $G_f$ in the pro-$p$ completion is equal to the pro-$p$ completion of their intersection, i.e.\ $\bar G_e\cap \bar G_f = \proP{G_e\cap G_f}$.
\end{enumerate}
Then $G$ is conjugacy $p$-separable.
\end{theorem}
The proof is in all respects a repetition of the argument in \cite{WZ10}, and we shall not reproduce it here. The difficulty lies in applying Theorem \ref{CpScombination} in the absence of any sledgehammer properties such as subgroup separability or double coset separability in the pro-$p$ world. Instead we must verify these properties for the specific cases involved in a particular application, and resist the temptation to attempt to prove too broad a result.

As an immediate consequence, when all the conditions on edge groups are trivial, we have:
\begin{clly}\label{freeprodCpS}
A free product of conjugacy $p$-separable groups is conjugacy $p$-separable.
\end{clly} 

We now prove a series of lemmas directed towards showing that the conditions of Theorem \ref{CpScombination} hold in the cases of Fuchsian groups and $p$-efficient graph manifolds. Many of the lemmas follow closely the analogous results for the profinite topology; where this is wholly or partly the case the result will be cited in brackets. 

In \cite{Nib92} Niblo uses the following `doubling trick' to deduce double-coset separability. The proof works just as well for the pro-$p$ topology, so we will use it to check condition 1 of Theorem \ref{CpScombination}.
\begin{theorem}[Niblo \cite{Nib92}]\label{doublingtrick}
Let $K,L$ be subgroups of $G$. Let $\tau$ denote the involution which swaps the two factors of $G\ast_L G$. If \gp{K,K^\tau} is $p$-separable in $G\ast_L G$ then the double coset $LK$ is $p$-separable in $G$.
\end{theorem}
\begin{proof}
Identical with the proof of \cite{Nib92}, Theorem 3.2.
\end{proof}
\begin{lem}
Let $\Sigma$ be a orientable surface, $G=\pi_1 \Sigma$ and let $D_1$, $D_2$ be maximal peripheral subgroups of $G$. Then the double coset $D_1 D_2$ is $p$-separable in $G$.
\end{lem}
\begin{proof}
By Proposition \ref{bdycptssep} we may assume $D_1\neq D_2$. Suppose that $D_1$, $D_2$ arise from boundary components $\bdy_1$, $\bdy_2$ of $\Sigma$ (possibly $\bdy_1=\bdy_2$). Choose a basepoint $x$ on $\bdy_1$; performing a conjugation we may assume that $D_1$ is generated by the homotopy class of the loop running around $\bdy_1$ based at $x$. Choose an immersed arc $\gamma$ joining $x$ to a point on $\bdy_2$ such that $D_2$ is generated by the homotopy class of the loop based at $x$ which runs along $\gamma$ to $\bdy_2$, once around $\bdy_2$, then back to $x$ along $\gamma$. In the case that $\bdy_2=\bdy_1$ choose $\gamma$ to be a loop based at $x$.

The finitely many self-intersections of $\gamma$ with itself give a finite collection of unbased loops in $\Sigma$; pass to a regular $p$-power degree cover $\pi:\tilde\Sigma\to\Sigma$ so that none of these loops lifts; such a cover exists since $\pi_1\Sigma$ is residually $p$. Furthermore, in the case when $\bdy_1=\bdy_2$, we can use the $p$-separability of $D_1=\pi_1(\bdy_1,x)$ to choose $\tilde\Sigma$ such that $\gamma$ is not congruent to any element of $D_1$ modulo $\pi_1 \tilde\Sigma$. Let $H\nsgp[p] G$ be the corresponding subgroup of $G$. Choose a lift $\tilde x$ of $x$ to $\tilde\Sigma$ to serve as a new basepoint. Then by construction $\gamma$ lifts to an {\em embedded} arc $\tilde\gamma$ in $\tilde\Sigma$ starting at $\tilde x$. Let the component of $\pi^{-1}(\bdy_1)$ containing $\tilde x$ be denoted $\tilde\bdy_1$, and the component of $\pi^{-1}(\bdy_2)$ containing the other endpoint of $\tilde\gamma$ be $\tilde\bdy_2$. Note that $\tilde\bdy_1\neq\tilde\bdy_2$ since if $\gamma$ is a loop, $\tilde\Sigma$ was constructed so that $\gamma$ does not lift either to a loop or to an arc with both endpoints on $\tilde\bdy_1$ (since such a lift would imply that $\gamma$ with congruent to $D_1$ modulo $H$). Then $D_1\cap H$ is generated by the loop $\tilde\bdy_1$ based at $\tilde x$, and $D_2\cap H$ is generated by the homotopy class of the loop based at $\tilde x$ which runs along $\tilde\gamma$ to $\tilde\bdy_2$, once around $\tilde\bdy_2$, then back to $\tilde x$ along $\tilde\gamma$.

%FIRST PICTURE
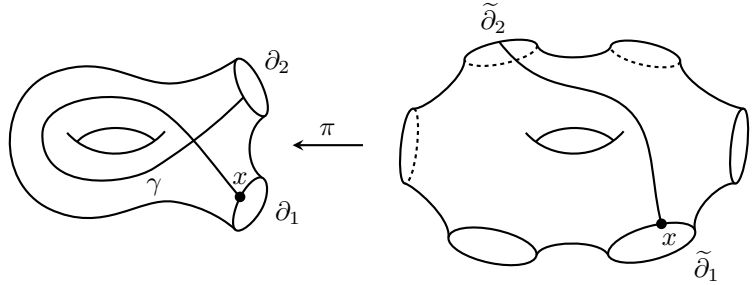
\begin{figure}[htp]
\centering
\begin{tikzpicture}[y=0.80pt, x=0.80pt, yscale=-1.000000, xscale=1.000000, inner sep=0pt, outer sep=0pt, scale=0.6]
\begin{scope}% layer1

  % path4039
  \path[draw=black,line join=miter,line cap=butt,thick]
    (387.7978,702.0095) .. controls (368.9649,698.1160) and (354.5571,689.1684) ..
    (355.4834,681.8693) .. controls (357.3734,666.2128) and (327.4709,640.5457) ..
    (321.1522,639.4868) .. controls (317.6853,638.3515) and (316.6873,622.7807) ..
    (318.7282,606.6215) .. controls (320.7691,590.4623) and (324.3385,580.3661) ..
    (328.3754,580.7432) .. controls (337.3415,581.4656) and (367.7462,557.1004) ..
    (367.9817,542.6359) .. controls (367.2055,537.9855) and (379.1658,531.5223) ..
    (394.7821,528.3293) .. controls (410.3983,525.1364) and (423.5414,526.4860) ..
    (424.5484,530.9458) .. controls (429.2991,544.9599) and (479.9018,541.9645) ..
    (481.9089,531.8525) .. controls (482.8300,526.9870) and (495.5877,525.8237) ..
    (510.4902,529.1332) .. controls (525.3927,532.4427) and (536.7968,538.9717) ..
    (536.0532,543.6434) .. controls (536.0906,557.3238) and (568.4686,578.4748) ..
    (585.7004,580.8693) .. controls (588.9878,581.4911) and (590.0564,597.6730) ..
    (588.0155,613.8322) .. controls (585.9746,629.9914) and (581.8665,640.3245) ..
    (578.9028,639.6285) .. controls (569.3826,637.6648) and (543.8037,659.5320) ..
    (547.6575,676.4509) .. controls (549.7712,683.9203) and (535.0816,694.1682) ..
    (516.3749,699.1698) .. controls (497.6682,704.1713) and (482.0271,701.9832) ..
    (479.9347,695.8021) .. controls (473.9315,683.3953) and (427.7400,684.0906) ..
    (424.0860,696.0080) .. controls (422.6976,703.3001) and (406.6933,705.9160) ..
    (387.7978,702.0095);

  % path4037-9
  \path[draw=black,line join=miter,line cap=butt,thick]
    (355.4892,681.8691) .. controls (356.5709,674.6806) and (372.8195,672.0714) ..
    (391.7150,675.9779) .. controls (410.6105,679.8844) and (425.1657,688.8339) ..
    (424.0841,696.0224);

  % path4035-9
  \path[draw=black,line join=miter,line cap=butt,thick]
    (479.9419,695.8138) .. controls (478.5570,688.6971) and (492.6529,678.3998) ..
    (511.3596,673.3982) .. controls (530.0663,668.3967) and (545.9538,669.9326) ..
    (547.6435,676.4911);

  % path4033-5
  \path[draw=black,line join=miter,line cap=butt,thick]
    (578.8899,639.6278) .. controls (574.8224,638.7575) and (574.0366,622.6860) ..
    (576.0774,606.5268) .. controls (578.1183,590.3676) and (581.7821,580.0504) ..
    (585.6745,580.8733);

  % path4031-53
  \path[draw=black,dash pattern=on 1.48pt off 1.48pt,line join=miter,line
    cap=butt,miter limit=4.00,thick] (536.0477,543.6604) .. controls
    (535.2129,548.4049) and (522.2837,549.7580) .. (507.3812,546.4485) .. controls
    (492.4787,543.1390) and (481.0746,536.6100) .. (481.9094,531.8655);

  % path4029-1
  \path[draw=black,dash pattern=on 1.48pt off 1.48pt,line join=miter,line
    cap=butt,miter limit=4.00,thick] (424.5432,530.9349) .. controls
    (425.4056,535.7147) and (413.4262,542.2673) .. (397.8100,545.4602) .. controls
    (382.1937,548.6531) and (368.8351,547.3667) .. (367.9727,542.5868);

  % path3866-9-7
  \path[draw=black,dash pattern=on 1.48pt off 1.48pt,line join=miter,line
    cap=butt,miter limit=4.00,thick] (328.3791,580.7470) .. controls
    (332.1772,582.0120) and (332.5725,597.6730) .. (330.5225,613.8468) .. controls
    (328.4725,630.0207) and (324.5664,640.6852) .. (321.1588,639.4834);

  % path4042-2
  \path[draw=black,line join=miter,line cap=butt,miter limit=4.00,thick] 
    (415.7354,599.7409) .. controls (435.4634,622.2289) and
    (472.5700,621.3245) .. (492.0335,598.5632);

  % path4044-1
  \path[draw=black,line join=miter,line cap=butt,miter limit=4.00,thick] 
    (423.6719,606.8880) .. controls (428.2447,600.7149) and
    (440.3089,596.1504) .. (452.9684,595.9847) .. controls (465.2885,595.8234) and
    (478.1884,599.8945) .. (484.1408,605.9644);

  % path3021
  \path[draw=black,line join=miter,line cap=butt,miter limit=4.00,thick] 
    (190.6577,541.7688) .. controls (195.5588,539.5618) and
    (203.9298,549.1476) .. (209.1782,560.1618) .. controls (214.6000,571.5400) and
    (216.7501,584.4314) .. (212.1544,586.7118) .. controls (196.9264,597.0623) and
    (196.7393,625.0267) .. (210.0125,634.9563) .. controls (215.5149,637.0262) and
    (215.1678,649.3988) .. (210.2207,659.0013) .. controls (204.4655,670.1721) and
    (193.3978,679.3185) .. (189.4139,676.0166) .. controls (185.0285,671.4666) and
    (167.1011,657.6105) .. (141.2985,653.9239) .. controls (114.4983,650.0947) and
    (86.7848,679.9549) .. (44.3309,659.5828) .. controls (1.0592,638.8182) and
    (4.6324,569.0019) .. (46.6871,552.7939) .. controls (91.0981,535.6777) and
    (113.9025,566.0695) .. (146.9438,560.2740) .. controls (166.8220,556.7874) and
    (181.2338,548.0802) .. (190.6577,541.7688) -- cycle;

  % path3028
  \path[draw=black,line join=miter,line cap=butt,miter limit=4.00,thick] 
    (190.9607,541.7954) .. controls (187.3899,543.2411) and
    (187.8611,555.6205) .. (192.9955,566.6419)node[xshift=14,yshift=11]{$\partial_2$} 
    .. controls (198.1563,577.7198) and
    (206.9186,588.2630) .. (212.3439,586.5568);

  % path3030
  \path[draw=black,line join=miter,line cap=butt,miter limit=4.00,thick] 
    (209.9479,634.9402) .. controls (205.0615,633.5834) and
    (197.2174,641.6226) .. (192.5532,651.3160) .. controls (187.5787,661.6540) and
    (185.7906,673.5153) .. (189.4971,676.0166);

  % path3821
  \path[draw=black,line join=miter,line cap=butt,thick,->,>=stealth]
    (288.8009,609.4824) -- (233.8984,609.4824) node[pos=0.5,yshift=6]{$\pi$};

  % path3823
  \path[draw=black,line join=miter,line cap=butt,thick]
    (394.8538,528.0387) node[xshift=-2,yshift=9]{$\widetilde\partial_2$} 
    .. controls (426.7117,570.1459) and (457.1216,558.6336) ..
    (485.0518,573.9801) .. controls (518.8076,592.5276) and (511.3677,632.1969) ..
    (521.4851,671.2796) node[circle, minimum size=4pt, fill=black]{} node[xshift=3,yshift=-6]{$x$}
    node[xshift=17,yshift=-16]{$\widetilde\partial_1$};

  % path3825
  \path[draw=black,line join=miter,line cap=butt,miter limit=4.00,thick] 
    (192.9718,650.1373) node[circle, minimum size=4pt, fill=black]{} node[xshift=0,yshift=7]{$x$} 
    node[xshift=18,yshift=-6]{$\partial_1$} 
    .. controls (178.7909,635.0549) and
    (154.1489,596.7294) .. (130.4332,581.4992) .. controls (101.3679,562.8334) and
    (39.8608,571.4272) .. (38.9027,600.3852) .. controls (37.3455,647.4468) and
    (105.1366,640.9580) .. (125.9012,627.5244)node[xshift=0,yshift=-7]{$\gamma$} 
    .. controls (164.3512,602.6492) and
    (183.8033,583.7721) .. (195.5659,571.9662);
    
      % path4042-2-0
  \path[draw=black,line join=miter,line cap=butt,miter limit=4.00,thick] 
    (58.1414,599.7409) .. controls (77.8694,622.2289) and
    (114.9759,621.3245) .. (134.4395,598.5632);

  % path4044-1-8
  \path[draw=black,line join=miter,line cap=butt,miter limit=4.00,thick] 
    (66.0779,606.8880) .. controls (70.6507,600.7149) and
    (82.7149,596.1504) .. (95.3744,595.9847) .. controls (107.6945,595.8234) and
    (120.5944,599.8945) .. (126.5468,605.9644);

\end{scope}

\end{tikzpicture}
\end{figure}
Note that $p$-separability of $(D_1\cap H)(D_2\cap H)$ in $H$ implies $p$-separability of $D_1 D_2$ in $G$; for the latter double coset is the union of finitely many translates of the former. We may now apply the `doubling trick'. Glue two copies $\tilde\Sigma$, $\tilde\Sigma^\tau$ of $\tilde\Sigma$ along $\tilde\bdy_1$ to obtain a surface $F$. The subgroup \gp{(D_2\cap H), (D_2\cap H)^\tau} of $\pi_1(F,\tilde x)= H\ast_{D_1\cap H} H$ is now the fundamental group of a certain subsurface $F'$ of $F$ whose boundary is an essential curve in $F$; specifically, take $F'$ to be a regular neighbourhood 
\[ {\cal N}(\gamma\cup \tilde\bdy_2\cup\gamma^\tau\cup \tilde\bdy_2^\tau)\]
Now $\gp{(D_2\cap H), (D_2\cap H)^\tau}=\pi_1(F',\tilde x)$ is $p$-separable in $H$ by Proposition \ref{surfacesplittingamalg}, so by Theorem \ref{doublingtrick} the double coset $(D_1\cap H)(D_2\cap H)$ is $p$-separable in $H$ and the proof is complete.

%SECOND PICTURE
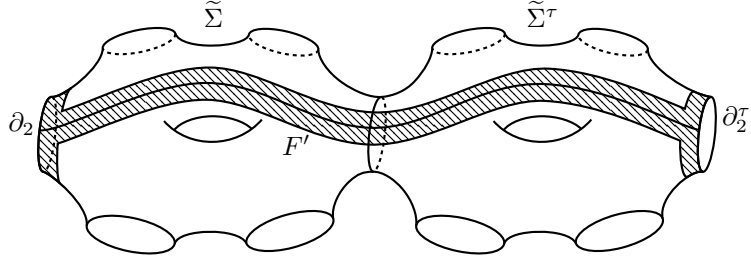
\begin{figure}[htp]
\centering
\begin{tikzpicture}[y=0.80pt, x=0.80pt, yscale=-1.000000, xscale=1.000000, inner sep=0pt, outer sep=0pt, scale=0.6]
\begin{scope}% layer1

  % path4037
  \path[draw=black,line join=miter,line cap=butt,thick]
    (326.8366,914.0603) .. controls (327.9183,906.8719) and (344.1669,904.2626) ..
    (363.0625,908.1691) .. controls (381.9580,912.0756) and (396.5132,921.0252) ..
    (395.4315,928.2136);

  % path4035
  \path[draw=black,line join=miter,line cap=butt,thick]
    (451.2894,928.0050) .. controls (449.9044,920.8884) and (464.0004,910.5911) ..
    (482.7071,905.5895) .. controls (501.4138,900.5879) and (517.3013,902.1239) ..
    (518.9910,908.6824);

  % path4033
  \path[draw=black,line join=miter,line cap=butt,thick]
    (550.2374,871.8190) .. controls (546.1698,870.9487) and (545.3840,854.8773) ..
    (547.4249,838.7181) .. controls (549.4658,822.5589) and (553.1295,812.2416) ..
    (557.0219,813.0646);

  % path4031
  \path[draw=black,dash pattern=on 1.48pt off 1.48pt,line join=miter,line
    cap=butt,miter limit=4.00,thick] (507.3952,775.8517) .. controls
    (506.5603,780.5962) and (493.6312,781.9493) .. (478.7287,778.6398) .. controls
    (463.8261,775.3303) and (452.4220,768.8013) .. (453.2569,764.0568);

  % path4029
  \path[draw=black,dash pattern=on 1.48pt off 1.48pt,line join=miter,line
    cap=butt,miter limit=4.00,thick] (395.8906,763.1261) .. controls
    (396.7530,767.9060) and (384.7737,774.4585) .. (369.1574,777.6515) .. controls
    (353.5412,780.8444) and (340.1826,779.5579) .. (339.3202,774.7780);

  % path3866-9
  \path[draw=black,dash pattern=on 1.48pt off 1.48pt,line join=miter,line
    cap=butt,miter limit=4.00,thick] (299.7265,812.9383) .. controls
    (303.5247,814.2032) and (303.9199,829.8642) .. (301.8700,846.0381) .. controls
    (299.8200,862.2120) and (295.9138,872.8765) .. (292.5063,871.6747);

  % path4042
  \path[draw=black,line join=miter,line cap=butt,miter limit=4.00,thick] (387.0829,831.9321) .. controls (406.8109,854.4201) and
    (443.9174,853.5158) .. (463.3809,830.7545);

  % path4044
  \path[draw=black,line join=miter,line cap=butt,miter limit=4.00,thick] (395.0194,839.0792) .. controls (399.5922,832.9061) and
    (411.6564,828.3416) .. (424.3158,828.1759) .. controls (436.6359,828.0147) and
    (449.5359,832.0857) .. (455.4883,838.1557);

  % path4039-6
  \path[draw=black,line join=miter,line cap=butt,thick]
    (102.2804,934.2008) .. controls (83.4475,930.3072) and (69.0397,921.3596) ..
    (69.9660,914.0606) .. controls (71.8559,898.4040) and (41.9534,872.7370) ..
    (35.6348,871.6781) .. controls (32.1679,870.5427) and (31.1699,854.9720) ..
    (33.2108,838.8128) .. controls (35.2517,822.6536) and (38.8211,812.5574) ..
    (42.8579,812.9345)node[xshift=-7,pos=0.1]{$\partial_2$} .. controls (51.8240,813.6568) and (82.2288,789.2916) ..
    (82.4643,774.8272) .. controls (81.6881,770.1768) and (93.6484,763.7135) ..
    (109.2647,760.5206) .. controls (124.8809,757.3277) and (138.0240,758.6773) ..
    (139.0310,763.1371) .. controls (143.7817,777.1511) and (194.3844,774.1558) ..
    (196.3915,764.0437)node[pos=0.5,yshift=12]{$\widetilde\Sigma$} .. controls (197.3126,759.1783) and (210.0703,758.0150) ..
    (224.9728,761.3245) .. controls (239.8753,764.6339) and (251.2794,771.1630) ..
    (250.5358,775.8347) .. controls (250.5547,782.7737) and (257.8580,793.6973) ..
    (268.7429,801.4381) .. controls (279.3177,808.9585) and (291.3546,813.7451) ..
    (299.9529,812.9975) .. controls (308.5512,812.2499) and (318.3984,808.5205) ..
    (325.5442,800.7407) .. controls (333.3730,792.2174) and (339.2162,781.7660) ..
    (339.3291,774.8272) .. controls (338.5530,770.1767) and (350.5133,763.7135) ..
    (366.1296,760.5206) .. controls (381.7458,757.3277) and (394.8889,758.6773) ..
    (395.8959,763.1370) .. controls (400.6466,777.1511) and (451.2493,774.1558) ..
    (453.2564,764.0437)node[pos=0.5,yshift=12]{$\widetilde\Sigma^\tau$} .. controls (454.1775,759.1783) and (466.9351,758.0150) ..
    (481.8377,761.3245) .. controls (496.7402,764.6339) and (508.1443,771.1630) ..
    (507.4007,775.8347) .. controls (507.4381,789.5151) and (539.8161,810.6660) ..
    (557.0479,813.0606) .. controls (560.3352,813.6823) and (561.4039,829.8643) ..
    (559.3630,846.0235) .. controls (557.3221,862.1827) and (553.2139,872.5157) ..
    (550.2503,871.8197) .. controls (540.7301,869.8560) and (515.1511,891.7232) ..
    (519.0050,908.6421) .. controls (521.1187,916.1115) and (506.4290,926.3594) ..
    (487.7223,931.3610) .. controls (469.0156,936.3626) and (453.3746,934.1744) ..
    (451.2822,927.9934) .. controls (445.2790,915.5866) and (399.0875,916.2818) ..
    (395.4335,928.1992) .. controls (394.0451,935.4914) and (378.0408,938.1073) ..
    (359.1453,934.2008)(359.1453,934.2008) .. controls (340.3124,930.3072) and
    (325.9046,921.3596) .. (326.8308,914.0606) .. controls (327.8254,905.8215) and
    (323.5378,894.1455) .. (316.1791,885.4034) .. controls (309.5543,877.5330) and
    (300.6907,871.5826) .. (292.9425,871.7489) .. controls (285.1944,871.9152) and
    (281.3155,873.2974) .. (274.4981,879.6281) .. controls (266.6326,886.9321) and
    (260.4220,899.5112) .. (262.1401,908.6421) .. controls (264.2538,916.1115) and
    (249.5642,926.3594) .. (230.8575,931.3610) .. controls (212.1508,936.3626) and
    (196.5097,934.1744) .. (194.4173,927.9934) .. controls (188.4141,915.5866) and
    (142.2226,916.2818) .. (138.5686,928.1992) .. controls (137.1802,935.4914) and
    (121.1759,938.1073) .. (102.2804,934.2008);

  % path4037-7
  \path[draw=black,line join=miter,line cap=butt,thick]
    (69.9717,914.0603) .. controls (71.0534,906.8719) and (87.3020,904.2626) ..
    (106.1976,908.1691) .. controls (125.0931,912.0756) and (139.6483,921.0252) ..
    (138.5666,928.2136);

  % path4035-3
  \path[draw=black,line join=miter,line cap=butt,thick]
    (194.4245,928.0050) .. controls (193.0396,920.8884) and (207.1355,910.5911) ..
    (225.8422,905.5895) .. controls (244.5489,900.5879) and (260.4364,902.1239) ..
    (262.1261,908.6824);

  % path4033-8
  \path[draw=black,line join=miter,line cap=butt,thick]
    (293.3725,871.8190) .. controls (289.3050,870.9487) and (288.5191,854.8773) ..
    (290.5600,838.7181) .. controls (292.6009,822.5589) and (296.2646,812.2416) ..
    (300.1571,813.0646);

  % path4031-5
  \path[draw=black,dash pattern=on 1.48pt off 1.48pt,line join=miter,line
    cap=butt,miter limit=4.00,thick] (250.5303,775.8517) .. controls
    (249.6954,780.5962) and (236.7663,781.9493) .. (221.8638,778.6398) .. controls
    (206.9613,775.3303) and (195.5572,768.8013) .. (196.3920,764.0568);

  % path4029-6
  \path[draw=black,dash pattern=on 1.48pt off 1.48pt,line join=miter,line
    cap=butt,miter limit=4.00,thick] (139.0257,763.1261) .. controls
    (139.8882,767.9060) and (127.9088,774.4585) .. (112.2926,777.6515) .. controls
    (96.6763,780.8444) and (83.3177,779.5579) .. (82.4553,774.7780);

  % path3866-9-0
  \path[draw=black,dash pattern=on 1.48pt off 1.48pt,line join=miter,line
    cap=butt,miter limit=4.00,thick] (42.8617,812.9383) .. controls
    (46.6598,814.2032) and (47.0550,829.8642) .. (45.0051,846.0381) .. controls
    (42.9551,862.2120) and (39.0490,872.8765) .. (35.6414,871.6747);

  % path4042-1
  \path[draw=black,line join=miter,line cap=butt,miter limit=4.00,thick] (130.2180,831.9321) .. controls (149.9460,854.4201) and
    (187.0525,853.5158) .. (206.5161,830.7545);

  % path4044-7
  \path[draw=black,line join=miter,line cap=butt,miter limit=4.00,thick] (138.1545,839.0792) .. controls (142.7273,832.9061) and
    (154.7915,828.3416) .. (167.4510,828.1759) .. controls (179.7711,828.0147) and
    (192.6710,832.0857) .. (198.6234,838.1557);

  % path3029-0
  \path[draw=black,line join=miter,line cap=butt,thick]
    (33.1001,839.2401) .. controls (52.6134,840.7482) and (117.0903,801.5403) ..
    (163.7878,802.0043) .. controls (210.4853,802.4682) and (244.2577,835.1885) ..
    (290.4841,837.2194) node[xshift=0,yshift=-14,pos=0.55]{$F'$}
    .. controls (336.7104,839.2503) and (375.8332,800.0778) ..
    (427.9629,802.9201) .. controls (480.0927,805.7624) and (532.3576,836.6075) ..
    (547.3631,838.8689) node[xshift=14,yshift=4]{$\partial_2^\tau$};

  % path4033-82
  \path[draw=black,
	pattern=north west lines,
        line join=miter,line cap=butt,thick] (557.0349,813.0626) .. controls
    (553.1295,812.2417) and (549.4658,822.5589) .. (547.4249,838.7181) .. controls
    (545.3840,854.8773) and (546.1698,870.9487) .. (550.2438,871.8194) .. controls
    (547.2233,871.1954) and (542.5729,872.9801) .. (537.8372,876.1773) .. controls
    (534.2638,875.7549) and (532.1050,861.4507) .. (532.6649,847.4229) .. controls
    (509.9760,838.2439) and (468.8015,818.3737) .. (427.4242,816.1177) .. controls
    (386.0469,813.8616) and (336.1717,852.4479) .. (289.9453,850.4170) .. controls
    (243.7190,848.3860) and (201.2298,815.5792) .. (163.2490,815.2018) .. controls
    (125.2683,814.8244) and (75.2994,840.6906) .. (47.3450,849.6207) .. controls
    (46.3665,863.6732) and (46.9929,878.2434) .. (48.8719,880.6603) .. controls
    (43.0075,875.4741) and (37.8594,872.0509) .. (35.6348,871.6781) .. controls
    (32.1679,870.5427) and (31.1699,854.9720) .. (33.2108,838.8128) .. controls
    (35.2517,822.6536) and (38.8211,812.5573) .. (42.8579,812.9344) .. controls
    (45.9423,813.1829) and (51.5637,810.4626) .. (57.7849,806.0648) .. controls
    (55.3178,807.4758) and (52.7403,814.6347) .. (50.6578,824.1408) .. controls
    (78.4355,815.3630) and (128.3972,789.5028) .. (166.3736,789.8802) .. controls
    (204.3501,790.2575) and (246.8436,823.0644) .. (293.0700,825.0953) .. controls
    (339.2963,827.1262) and (388.8258,788.5211) .. (430.5488,790.7960) .. controls
    (472.2718,793.0709) and (513.7398,813.1983) .. (536.3199,822.3293) .. controls
    (538.4295,814.0001) and (541.0651,808.3218) .. (543.3386,808.7547) .. controls
    (548.1970,810.9269) and (553.0335,812.5027) .. (557.0349,813.0626) -- cycle;

\end{scope}

\end{tikzpicture}
\end{figure}
\end{proof}
\begin{clly}\label{FuchsDblCoset}
Let $G$ be the fundamental group of a 2-orbifold $O$; assume $G$ is residually $p$ and that $O$ is orientable if $p\neq 2$. Let $D_1$, $D_2$ be maximal peripheral subgroups of $G$. Then the double coset $D_1 D_2$ is $p$-separable in $G$.
\end{clly}
\begin{proof}
As $G$ is residually $p$, there is a regular index $p$ cover of $O$ which is an orientable surface $\Sigma$. If $H=\pi_1 \Sigma$ then $(D_1\cap H)(D_2\cap H)$ is $p$-separable in $H$. As noted above, this implies that $D_1 D_2$ is $p$-separable in $G$.
\end{proof}
\begin{clly}\label{SFSDblCoset}
Let $G$ be the fundamental group of a \SFS{} $M$ with non-empty boundary; assume $G$ is residually $p$ and let $D_1$, $D_2$ be maximal peripheral subgroups of $G$. Then the double coset $D_1 D_2$ is $p$-separable in $G$.
\end{clly}
\begin{proof}
Again it suffices to pass to a regular $p$-cover; because $G$ is residually $p$, $G$ admits a regular $p$-cover of the form $\Sigma\times\sph{1}$, where $\Sigma$ is an orientable surface. If $\pi: \Sigma\times\sph{1}\to\Sigma$ is the projection, then 
\[ D_1 D_2 = \pi_\ast(D_1)\pi_\ast(D_2)\times \Z\]
so the result follows.
\end{proof}

\begin{lem}[cf Lemma 6.3 of \cite{Wilk16}]\label{orbbdypro-p}
Let $O$ be a hyperbolic 2-orbifold with non-empty boundary and no reflector curves. Let $\bdy_1, \bdy_2$ be curves representing components of $\bdy O$. Suppose \ofg[O] is residually $p$. Let $\Gamma=\proP{\ofg[O]}$, and let $\Delta_i$ be the closure in $\Gamma$ of $\pi_1 \bdy_i$. Then for $\gamma_i\in \Gamma$, either $\Delta_1^{\gamma_1}\cap \Delta_2^{\gamma_2}=1$ or $\bdy_1=\bdy_2$ and $\gamma_2 \gamma_1^{-1}\in \Delta_1$.
\end{lem}
\begin{proof}
By conjugating by $\gamma_1^{-1}$ we may assume that $\gamma_1=1$; drop the subscript on $\gamma_2=\gamma$. Note that $\Delta_1\cap \Delta_2^\gamma$ is torsion-free, so it is sufficient to pass to a finite index subgroup $\Gamma'$ and show that $\Delta_1\cap \Delta_2^\gamma \cap \Gamma' = 1$. Suppose that this intersection is non-trivial.

Because $O$ is hyperbolic and \ofg[O] is residually $p$, there is some regular cover $O'$ of $O$ with degree a power of $p$ with more than two boundary components; then given any pair of boundary components, \ofg[O'] has a decomposition as a free product of cyclic groups, among which are the two boundary components. Let $\Gamma'$ be the corresponding finite index normal subgroup of $\Gamma$. Note that for some set $\{h_i\}$ of coset representatives of $\Gamma'\cap \ofg[O]$ in \ofg[O] (which give coset representatives of $\Gamma'$ in $\Gamma$), each $\Delta_2 ^{h_i}$ is the closure of the fundamental group of a component of $\bdy O'$; so set $\Delta_3=\Delta_2^{h_i}$ where $\gamma=h_i \gamma'$ for some $\gamma'\in \Gamma'$. Furthermore, if two boundary components of $O$ are covered by the same boundary component $O'$, then they must have been the same boundary component of $O$; that is, if $\Delta_1\cap \Gamma'=\Delta_3\cap \Gamma'$, then $\Delta_1=\Delta_3$. 

Now the intersections of $\Delta_1, \Delta_3$ with $\Gamma'$ are free factors; that is, \[\Gamma'=(\Delta_1\cap \Gamma')\amalg (\Delta_3\cap \Gamma')\amalg \Phi\]
where $\Phi$ is a free pro-$p$ product of cyclic groups (unless $\Delta_1=\Delta_3$, when $\Gamma'=(\Delta_1\cap \Gamma')\amalg F$). Let $T$ be the standard graph for this free product decomposition of $\Gamma'$. Then $\Delta_1\cap \Gamma'=\Gamma'_v, \Delta_3\cap \Gamma' = \Gamma'_w$ for vertices $v,w\in T$. The action on $T$ is 0-acylindrical because all edge stabilisers are trivial; so for $\gamma'\in \Gamma'$, the intersection 
\[\Delta_1\cap \Delta_3^{\gamma'}\cap \Gamma'= \Gamma'_v\cap \Gamma'_{{\gamma'}^{-1}\cdot w}\] can only be non-trivial if $v={\gamma'}^{-1}\cdot w$, so that $\Delta_1\cap \Gamma'=\Delta_3\cap \Gamma'$ (hence $\Delta_1=\Delta_3$) and $\gamma'\in \Delta_1$.

We have reduced to the case where $\Delta_2^{h_i}=\Delta_1$ and must show that $D_1=D_2$ and $h_i\in D_2$, for then our original element $h_i \gamma'= \gamma\in \Gamma$ is in $\Delta_2$. The intersection of two distinct peripheral subgroups of $\ofg[O]$ is trivial, and peripheral subgroups coincide with their normalisers in \ofg[O]. Suppose that $D_2^{h_i}\neq D_2$. We can pass to a regular $p$-cover of $O$ to which $h_i$ does not lift, and with more than two boundary components; so that the lifts of $D_2^{h_i}$ and $D_2$ are distinct free factors, hence their closures in the pro-$p$ completion have trivial intersection. But $\Delta_1\cap \Delta_2\neq 1$ by assumption, so that in fact $D_2^{h_i}= D_2$ hence $h_i\in \Delta_2=\Delta_1$ as required.
\end{proof}

\begin{lem}[Proposition 5.4 of \cite{WZ10}]\label{bdyintersectpro-p}
Let $L$ be a \SFS{} with non-empty boundary with hyperbolic base orbifold $O$. Suppose that $\pi_1 L$ is residually $p$. Let $\Lambda=\proP{\pi_1 L}$, and $Z$ be the subgroup of $\pi_1 L$ generated by a regular fibre. Let $\Delta_1$, $\Delta_2$ be peripheral subgroups of $H$; that is, conjugates in $H$ of the closure of peripheral subgroups of $\pi_1 L$. Then $\Delta_1\cap \Delta_2=\bar Z$ unless $\Delta_1=\Delta_2$, where $\bar Z$ is the closure of $Z$ in $\Lambda$.
\end{lem}
\begin{proof}
Identical to loc.\ cit.\, given the previous lemma.
\end{proof}

For the next two propositions we use the following notation. Let $G$ be the fundamental group of a $p$-efficient graph manifold, with graph of groups decomposition $(X,G_\bullet)$. Let $\Gamma=\Pi_1(\proP{\cal G})$ be the pro-$p$ completion of $G$. Let $S(\proP{\cal G})$ be the standard tree for this graph of pro-$p$ groups. For a vertex group $G_v$ of $\cal G$, let $Z_v$ be the subgroup generated by its regular fibre (the `canonical fibre subgroup'). Let $\bar Z_v$ be the closure in $\Gamma_v=\proP{G_v}$ and extend this notation to all vertex groups of $S(\proP{\cal G})$ by the conjugation action. 

\begin{lem}
Let $e=[v,w]$ be an edge of $S(\proP{\cal G})$. Let $Z_v, Z_w$ be the canonical fibre subgroups of $G_v,G_w$ respectively. Then $\gp{\bar Z_v, \bar Z_w}\nsgp[p] \Gamma_e$, and so $\bar Z_v\cap \bar Z_w =1$.
\end{lem}
\begin{proof}
After a conjugation in $\Gamma$, we may assume that $e$ is an edge in the standard graph of the abstract fundamental group $G$, i.e. $\Gamma_e$ is the closure in $\Gamma$ of a peripheral subgroup of some $G_v$. Elementary calculations show that if two elements of $\Z^2$ generate an index subgroup $p^r m$ subgroup of $\Z^2$, where $m$ is coprime to $p$, then they generate a subgroup of any $p$-group quotient of $\Z^2$ of index dividing $p^r$; hence generate an index $p^r$ subgroup of $\Z[p]^2$. The result follows.
\end{proof}

\begin{prop}[cf Proposition 6.8 of \cite{Wilk16}, Lemma 5.5 of \cite{WZ10}]\label{JSJacylpro-p}
Let $M$ be a $p$-efficient graph manifold in which all \SFS{}s have hyperbolic base orbifold. Then the action of $\Gamma=\proP{\pi_1 M}$ on the standard graph $S(\proP{\cal G})$ is 2-acylindrical.
\end{prop}
\begin{rmk}
The condition on the base orbifolds is automatic when $p\neq 2$; in general it may be achieved by passing to an index 2 cover.
\end{rmk}
\begin{proof}
Take a path of length 3 in $S(\proP{\cal G})$ consisting of edges $e_0, \ldots, e_2$ joining vertices $v_0,\ldots, v_3$. By Lemma \ref{bdyintersectpro-p}, $\Gamma_{e_0}\cap \Gamma_{e_1}=\bar Z_{v_1}$ and $\Gamma_{e_1}\cap \Gamma_{e_2}=\bar Z_{v_2}$; but $\bar Z_{v_1}\cap \bar Z_{v_2}$ is trivial by the previous lemma. So $\bigcap_{i=0}^2 \Gamma_{e_i}$ is trivial as required.
\end{proof}

\begin{prop}\label{FuchsCp-D}
Let $O$ be a hyperbolic 2-orbifold with non-empty boundary and no reflector curves. Let $G=\ofg[O]$ and suppose $G$ is residually $p$. Let $D=\gp{l}$ be the fundamental group of a boundary component of $O$. Then $D$ is conjugacy $p$-distinguished in $G$. 
\end{prop}
\begin{proof}
First suppose that $D$ is a free factor of $G$, say $G=D\ast G'$. Suppose that $g\in G$ is not conjugate in $G$ to any power of $l$. Write $g$ as a reduced word
\[ g= g_1 d_1 g_2\ldots g_n d_n\]
where $g_i \in G'$, $d_i\in D$ are all non-trivial except perhaps $g_1, d_n$. We may ensure $g_1\neq 1$ by conjugating by $d_1$. Since $g$ is not conjugate into $D$, at least one of the following occurs:
\begin{itemize}
\item $n$ is odd
\item $d_n\neq 1$
\item for some $i$, $g_i\neq g_{n+1-i}^{-1}$
\item for some $i\neq n/2$, $d_i\neq d_{n-i}^{-1}$
\end{itemize}
since if all the above fail, we have expressed $g$ as a conjugate of $d_{n/2}$. By uniqueness of reduced forms, no element whose reduced form has any of the above properties can be conjugate into $D$; for writing any $h\in G$ as a reduced word, $h^{-1}dh$ is already written as a reduced word, having none of the above properties.

Now $G$ is residually $p$, so we may find finite $p$-group quotients $D\to P_1$, $G'\to P_2$ such that no non-trivial $d_i$  or $g_i$ vanishes under the quotient map, and so that any of the properties from the above list are preserved in the quotient; that is, if $\phi:G\to P_1\ast P_2$ is the quotient map, $\phi(g)$ is a reduced word in $P_1\ast P_2$, which has one of the above properties, hence is not conjugate into $P_1$. Since $P_1$ is finite and $P_1\ast P_2$ is conjugacy $p$-separable, there is a $p$-group quotient $\psi:P_1\ast P_2\to Q$ such that $\psi\phi(g)$ is not conjugate into $\psi\phi(D)=\psi(P_1)$; hence $D$ is conjugacy $p$-distinguished in $G$. 

We now deal with the general case. Let $g\in G$ and suppose that $\gamma^{-1} g\gamma=l^\alpha\in \bar D$ for some $\gamma\in\proP{G},\alpha\in\Z[p]$. Note that $g$ is infinite order. Let $F\nsgp[p]G$ represent a regular $p$-power degree cover of $O$ with more than one boundary component, so that $D\cap F$ is a free factor of $F$. Note that $\gamma=h\delta$ for some $h\in G,\delta\in\bar F$. For some $n=p^r$, we have $g^n\in F$; and 
\[ \delta^{-1} (h^{-1}g^n h) \delta = \gamma^{-1} g^n \gamma = l^{n\alpha}\in \overline{F\cap D}\]
By the first part, since $F\cap D$ is conjugacy $p$-distinguished in $F$ and $\delta\in\bar F$, there exists some $f\in F$ such that $f^{-1} (h^{-1}g^n h) f\in F\cap D$. Thus $g'= (hf)^{-1}g(hf)$ is a parabolic element of $G$, some power of which lies in $D$; and since parabolic subgroups of a Fuchsian group either intersect trivially or are equal, it follows that $g'\in D$ so that $g$ is conjugate into $D$ as required.
\end{proof}
Recall for the following that the boundary of a 2-orbifold is not necessarily the same as the boundary $\bdy_{\rm top}$ of the underlying surface. An orbifold with boundary is locally modelled on quotients of open subsets of the upper half-plane by group actions, and boundary points of the orbifold come from boundary points of the upper half-plane. Some portions of $\bdy_{\rm top}$ may indeed be part of the orbifold boundary; however some of $\bdy_{\rm top}$ may be included in the singular locus as `reflector' curves. The isotropy group of an interior point of a reflector curve is $\Z/2$. The endpoints of a reflector curve may have `corner reflector' points whose isotropy subgroup is dihedral. Alternatively an endpoint of a reflector curve may again have isotropy group $\Z/2$, the local model for such a point being the upper half-plane modulo a reflection in the $y$-axis. Since reflections are order 2, when $p\neq 2$ reflector curves do not arise in an orbifold with residually $p$ fundamental group. When they do arise, there is a canonical `reflectorless' index 2 cover of the orbifold with no reflector curves; corner reflectors become cone points in this cover. An orbifold is said to be orientable if its underlying surface is orientable.

\begin{lem}\label{reflectors2-dist}
Let $O$ be a hyperbolic 2-orbifold with residually 2 fundamental group $G$ and let $\rho$ be  a reflector curve of $O$ with isotropy group $\Z/2=\gp{\tau}$. Then $\tau$ is conjugacy 2-distinguished in $G$.
\end{lem}
\begin{proof}
First consider the reflectorless degree 2 cover $O'$ of $O$ obtained by doubling along reflector curves, and the corresponding index 2 subgroup $G'$ of $G$. The order 2 elements of $G$ which do not lie in $G'$ are precisely the conjugates of reflector elements; cone points in $O$ lift to $O'$, and the intersection of each isotropy group of a corner reflector with $G'$ is precisely its rotation subgroup. It thus suffices to distinguish $\tau$ from the other reflector elements. So let $\rho'$ be a different reflector curve of $O$, with isotropy group $\gp{\tau'}$. Take a quotient of the orbifold $O$ by collapsing the complement of a neighbourhood of the boundary component of $\bdy_{\rm top}(O)$ containing $\rho$. If this component did not contain $\rho'$, then in this quotient group $\tau'$ has become trivial; so the canonical reflectorless cover of this quotient orbifold yields a quotient $\Z/2$ distinguishing $\tau$ from $\tau'$. Pass to a further quotient by abelianising the isotropy group of each corner reflector to obtain a copy of $\Z/2\oplus\Z/2$, where the two incident reflector curves generate the two factors. We are left with a right-angled Coxeter group in which $\tau,\tau'$ form part of a standard generating set; they thus have distinct images in first $\Z/2$-homology, hence are not conjugate in this quotient of $G$. This completes the proof. 
\end{proof}

\begin{defn}\label{Hakenorb}
A {\em hierarchical} (2-)orbifold will mean any 2-orbifold which is not on the following list:
\begin{itemize}
\item a sphere or projective plane with at most 3 cone points; or
\item a disc or M\"obius band, with $\bdy_{\rm top}$ composed entirely of reflector curves and with at most one cone point and at most three corner reflectors.
\end{itemize}
\end{defn}
Notice that in the above definition the reflectorless cover of any hierarchical orbifold is also hierarchical.

The reason for this definition is that all hierarchical orbifolds $O$ admit a `hierarchy' of the following type. If the orbifold has any genuine boundary curves or arcs, then cutting along arcs with both endpoints on a genuine boundary curve/arc (i.e.\ along an interval with trivial fundamental group) or along an arc with one endpoint on a genuine boundary curve/arc and the other endpoint on a reflector curve (i.e.\ along the quotient of an interval by a reflection, a 1-orbifold with fundamental group $\Z/2$) allows us to decompose the orbifold fundamental group into copies of \Z, $\Z/2\ast \Z/2$, and $p$-groups glued along copies of $\Z/2$ or the trivial group. Note that in this case the reflectorless index 2 subgroup is correspondingly decomposed as a free product of \Z{} and a collection of $p$-groups.

When the entirety of $\bdy_{\rm top}$ is composed of reflector curves, and $O$ is not on the above list, one may still obtain a hierarchy. We will not in fact use this hierarchy in the sequel, but it gives more consistency to the definition of `hierarchical'. The first stage in the hierarchy is obtained as follows. If $O$ is a disc, or M\"obius band with reflector boundary and at least four corner reflectors, let $l$ be an embedded 1-orbifold whose endpoints lie on the reflector curve such that at least two corner reflectors lie on either side of $l$; note that the orbifold fundamental group of $l$ is a copy of $\Z/2\ast \Z/2$ along which $G$ splits. If $O$ is a cylinder with reflector boundary, choose an embedded 1-orbifold $l$ with one endpoint on each reflector curve; again $G$ splits over $\Z/2\ast \Z/2=\ofg[l]$. Otherwise choose an essential simple closed curve $l$ on $O$ which does not pass through any cone points; such a curve exists for any orbifold other than those appearing in the above list. 

\begin{theorem}\label{FuchNotTri}
Let $G=\ofg[O]$ be a residually $p$ Fuchsian group, where $O$ is a hyperbolic 2-orbifold that is orientable when $p\neq 2$. Suppose further that $O$ is hierarchical. Then $G$ is conjugacy $p$-separable.
\end{theorem}
\begin{proof}
We note that each splitting of $G$ given by the above hierarchies satisfies the conditions of Theorem \ref{CpScombination}. First consider the case when $O$ has no reflector curves; this covers all cases when $p\neq 2$. When $O$ has (genuine) boundary the result follows from Corollary \ref{freeprodCpS} since then we may decompose $G$ as a suitable free product of free groups and $p$-groups. Otherwise we have a splitting of $G$ along a simple closed curve as an amalgamated free product or HNN extension of Fuchsian groups with (genuine) boundary, which are conjugacy $p$-separable. Passing to a regular cover of $O$ which is a surface, the splittings along lifts of $l$ are $p$-efficient by Proposition \ref{surfacesplitting}; hence the splitting of $G$ is $p$-efficient. The action on the standard pro-$p$ tree of the splitting is 1-acylindrical by Lemma \ref{orbbdypro-p}. The remaining conditions 1, 2, 3 in Theorem \ref{CpScombination} hold by Corollary \ref{FuchsDblCoset}, Proposition \ref{FuchsCp-D}, and Lemma \ref{orbbdypro-p} respectively. Hence we may apply Theorem \ref{CpScombination} to conclude that $G$ is conjugacy $p$-separable.

Now let $p=2$ and suppose that $O$ has reflector curves. Let $O'$ be the canonical reflectorless degree 2 cover of $O$ obtained by doubling $O$ along its reflector curves and replacing any corner reflectors by cone points. Let $G'=\ofg[O']$. Note that $O'$ is a hierarchical orbifold. Let $g_1, g_2\in G$ be conjugate in the pro-2 completion \proP[2]{G}. If $g_1\in G'$ then $g_1$ is conjugacy 2-distinguished in $G'$, hence in $G$ by Lemma \ref{Lem1Ste} below so we are done. So suppose $g_1$  (hence $g_2$) is in $G\smallsetminus G'$. If $g_1$ has order 2 then since the only order 2 elements of $G\smallsetminus G'$ are in isotropy groups of reflector curves we are done by Lemma \ref{reflectors2-dist}.

So suppose $g_1$ is infinite order. Let $\gamma\in \proP[2]{G}$ be such that $g_1=g_2^\gamma$. Conjugating $g_2$ by an element $\tau\in G'$ we may assume that $\gamma$ lies in \proP[2]{G'}. Then $g_1^2$ is conjugate in \proP[2]{G'} to $g_2^2$; since $G'$ is conjugacy 2-separable, they are conjugate in $G'$. After a conjugation by an element of $G'$ we may thus assume $g_1^2=g_2^2$. Any infinite order element of a Fuchsian group has at most two square roots, differing by a reflection. So either $g_1=g_2$ as required or one of $g_1$, $g_2$ is orientation preserving and the other is orientation reversing. in the latter case have different images under the orientation homomorphism $G\to\Z/2$, and so cannot be conjugate in \proP[2]{G}. This concludes the proof.  
\end{proof}

The extension of this to all Fuchsian groups does not follow immediately, since conjugacy separability is not a commensurability invariant (see \cite{Gor86,CZ09,MM12}). In what follows we remind that reader that `open subgroups $H\sbgp[p] G$' are those subgroups of $G$ with index a power of $p$ such that $H$ contains some normal subgroup of $G$ with index a power of $p$. Note that $G$ induces the full pro-$p$ topology on such an $H$, and remark that not all subgroups with index a power of $p$ are necessarily open (for instance a symmetric group $S_{p-1}\sbgp[p] S_p$ for $p\geq 5$).

\begin{lem}[cf Lemma 1 of \cite{Ste70}]\label{Lem1Ste}
Let $g\in G$, and suppose that $H\sbgp[p] G$ is open in $G$ and contains $g$. If $g$ is conjugacy $p$-distinguished in $H$, then it is conjugacy $p$-distinguished in $G$.
\end{lem}
\begin{proof}
If $\{g_1,\ldots, g_n\}$ is a complete set of right coset representatives of $H$ in $G$, then 
\[ g^G = \bigcup_{i=1}^n (g^H)^{g_i} \]
where superscripts denote conjugation. By assumption $g^H$ is closed in $H$, hence in $G$; thus since $g^G$ is finite union of translates of $g^H$, the conjugacy class $g^G$ is closed in $G$ and $g$ is conjugacy $p$-distinguished in $G$.
\end{proof}
\begin{prop}[Theorem 3.9 of \cite{Ste72}]\label{inforderCpD}
Let $G$ be a group containing a free group or a surface group $F\nsgp[p] G$. Then elements of infinite order in $G$ are conjugacy $p$-distinguished.
\end{prop}
\begin{proof}
The proof is identical with that of \cite{Ste72}, noting that all finite index subgroups constructed there are open and have index a power of $p$ in the present situation.
\end{proof}
\begin{lem}[Lemma 3.8 of \cite{Ste72}]\label{VirtAbCpS}
Let $G$ be a group, $A\nsgp[p] G$. Suppose that $A$ is a residually $p$ Abelian group. Then $G$ is conjugacy $p$-separable.
\end{lem}
\begin{proof}
Again the proof in \cite{Ste72} works with no modification.
\end{proof}
\begin{theorem}\label{FuchsCpS}
Let $O$ be a 2-orbifold, and suppose that $G=\ofg[O]$ is residually $p$ and that $O$ is orientable when $p\neq 2$. Then $G$ is conjugacy $p$-separable.
\end{theorem}
\begin{proof}
If $O$ is not hyperbolic then $G$ has an abelian subgroup $A\nsgp[p] G$, so that we are done by Lemma \ref{VirtAbCpS}. By Theorem \ref{FuchNotTri} we have reduced to the case of those orbifolds appearing in the statement of Theorem \ref{FuchNotTri}. Take $g\in G$; we must show that $g$ is conjugacy $p$-distinguished. By Proposition \ref{inforderCpD}, without loss of generality $g$ is finite order, say $p^n$. Since $G$ is residually $p$, there are arbitrarily large $p$-group quotients $\phi:G\to P$ into which \gp{g} injects. Choose $|P|$ sufficiently large that $H =\phi^{-1}(\gp{\phi(g)})$ has rational Euler characteristic at most $-3$. Considering Definition \ref{Hakenorb} we see that all non-hierarchical 2-orbifolds have Euler characteristic strictly greater than -3. So $H$ is the fundamental group of a hierarchical 2-orbifold. Then $H$ is conjugacy $p$-separable by Theorem \ref{FuchNotTri}, so $g$ is conjugacy $p$-distinguished in $H$. Note that $H\sbgp[p] G$ is an open subgroup of $G$ containing $g$, hence $g$ is conjugacy distinguished in $G$ by Lemma \ref{Lem1Ste}. So $G$ is conjugacy $p$-separable.
\end{proof}
Given Theorem \ref{FuchsCpS} the next two theorems follow from similar results in \cite{Mar06} by simply checking that all finite-index subgroups constructed can be chosen to be normal of index a power of $p$.
\begin{theorem}[Theorem 3.7 of \cite{Mar06}]
Let $G$ contain an orientable surface subgroup $\pi_1 \Sigma\nsgp[p] G$. Then $G$ is $p$-conjugacy separable.
\end{theorem}
\begin{lem}[Lemma 4.2 of \cite{Mar06}]\label{p-centralext}
Let $H$ be a group containing a normal $p$-power index orientable surface subgroup. Suppose $G$ is a central extension of $H$ by a finite $p$-group. Then $G$ contains a normal orientable surface subgroup of index a power of $p$ and hence is conjugacy $p$-separable.
\end{lem}
\begin{theorem}[cf Martino \cite{Mar06}]
Let $G$ be the fundamental group of a \SFS{} which has hyperbolic base orbifold. Assume that $G$ is residually $p$. Then $G$ is conjugacy $p$-separable. 
\end{theorem}
\begin{proof}
Suppose first that $p\neq 2$ and let $g, g'$ be non-conjugate elements of $G=\pi_1 M$. Let $h$ denote the homotopy class of a regular fibre of $M$ and let $O$ be the quotient orbifold of $M$, so that we have a central extension 
\[1\to \gp{h} \to G \to \ofg[O]\to 1\] 
If the images of $g,g'$ in \ofg[O] are not conjugate, we are done by Theorem \ref{FuchsCpS}. So suppose $g,g'$ are conjugate in \ofg[O]; after a conjugacy we may assume that $g'=gh^n$ for some $n$. Choose some $k$ such that $p^k>|n|$ and consider the quotient $\phi:G\to G'=G/\gp{h^{p^k}}$. Note that centralisers in Fuchsian groups are cyclic, so that the pre-image of the centraliser of $g$ in \ofg[O] is a copy of $\Z{}^2$; so if $x\in G'$ conjugates $\phi(g)$ to $\phi(gh^m)$ for some $m$, then in fact $x$ commutes with $\phi(g)$ in $G'$, and hence  $\phi(g')$ is not conjugate to $\phi(g)$ in $G'$. By Lemma \ref{p-centralext} $G'$ is conjugacy $p$-separable and we are done.

Now let $p=2$; the difference here is that $O$ may be non-orientable. Let $G^+$ be the index 2 subgroup of $G$ consisting of elements which centralise $h$. If $g\in G^+$ then $g$ is conjugacy $p$-distinguished in $G^+$, hence in $G$ by Lemma \ref{Lem1Ste}. So suppose $g\in G\smallsetminus G^+$ and let $g'\in G$ be a non-conjugate of $g$. Again it suffices to deal with the case $g'=gh^n$. Now, since $g^{-1} h g = h^{-1}$, $g$ is conjugate to $gh^{2k}$ for all $k$; so $n$ is odd. Consider the quotient $\phi:G\to G'=G/\gp{h^{2}} $, which is conjugacy 2-separable. Suppose $x\in G'$ conjugates $\phi(g)$ to $\phi(g')$. Again the centraliser of the image of $G$ in \ofg[O] is cyclic, and the preimage of this group is a copy of $\Z\times\Z/2$ containing $x$. Hence $\phi(g), \phi(g')$ are not conjugate and we are done.
\end{proof}
The restriction to hyperbolic base orbifold in the above theorem was necessary to exclude problems with the geometry Nil, as the following example shows. Note that the three remaining Seifert fibred geometries (\sph{3}, $\sph{2}\times\R$, and $\E^3$) have no such issues as all these groups are finite or virtually abelian and are easily dealt with. 
\begin{example}
We claim that the Heisenberg group $G={\cal H}_3(\Z)$ is not conjugacy $p$-separable for any prime $p$. Suppose $p\neq 2$, the $p=2$ case being similar. We have a presentation 
\[ G = \langle x,y,h\,\big|\, [x,y]=h\text{ central}\rangle\]
By direct calculation, $x^2$ is not conjugate to $x^2 h$; however for any $n$, 
\[ y^{-n}x^2 y^n = x^2 h^{2n}\]
In any $p$-group quotient $\phi:G\to P$, we have $\phi(x^2 h) = \phi(x^2 h^{2n})$ for some $n$, so that the image of $x^2 h$ is always conjugate to the image of $x^2$, proving the claim. Note that the congruence quotients exhibit that ${\cal H}_3(\Z)$ is indeed residually $p$. See \cite{Iva04} for a characterisation of conjugacy $p$-separable nilpotent groups.
\end{example}
\begin{theorem}
Let $G$ be the fundamental group of a $p$-efficient graph manifold in which all \SFS{}s have hyperbolic base orbifold. Then $G$ is conjugacy $p$-separable.
\end{theorem}
\begin{proof}
The vertex groups are conjugacy $p$-separable by the previous result. By Proposition \ref{JSJacylpro-p}, the action on the standard pro-$p$ tree of this splitting is 2-acylindrical. Condition 1 of Theorem \ref{CpScombination} holds by Corollary \ref{SFSDblCoset}. Condition 2 holds by Proposition \ref{FuchsCp-D} since an element of a vertex group is conjugate into the boundary if and only if its image in the Fuchsian quotient is conjugate into the boundary. Condition 3 holds by Lemma \ref{bdyintersectpro-p}. Hence Theorem \ref{CpScombination} applies and $G$ is conjugacy $p$-separable.
\end{proof}
Since by Section 5.1 of \cite{AF13}, any graph manifold has a finite-sheeted cover of the above type, Theorem \ref{introCpS} follows immediately.

\nocite{BCR14}
\nocite{RZ96}
\bibliographystyle{plain}
\bibliography{VPP}

\begin{thebibliography}{10}

\bibitem{AGM13}
Ian Agol, Daniel Groves, and Jason Manning.
\newblock {The virtual Haken conjecture}.
\newblock {\em {Doc. Math}}, 18:1045--1087, 2013.

\bibitem{AF13}
Matthias Aschenbrenner and Stefan Friedl.
\newblock {\em {3-manifold groups are virtually residually $p$}}.
\newblock {American Mathematical Society}, 2013.

\bibitem{AFW15}
Matthias Aschenbrenner, Stefan Friedl, and Henry Wilton.
\newblock {\em {3-manifold groups}}.
\newblock {European Mathematical Society}, 2015.

\bibitem{BCR14}
Martin~R. Bridson, Marston~DE Conder, and Alan~W. Reid.
\newblock {Determining Fuchsian groups by their finite quotients}.
\newblock {\em {arXiv preprint arXiv:1401.3645}}, 2014.

\bibitem{CZ09}
Sheila~C. Chagas and Pavel~A. Zalesskii.
\newblock {Finite index subgroups of conjugacy separable groups}.
\newblock In {\em {Forum Mathematicum}}, volume~21, pages 347--353, 2009.

\bibitem{chat94}
Zo{\'{e}}~Maria Chatzidakis.
\newblock {Some remarks on profinite {HNN} extensions}.
\newblock {\em {Israel Journal of Mathematics}}, 85(1-3):11--18, 1994.

\bibitem{DdSMS03}
John~D. Dixon, Marcus~PF du~Sautoy, Avinoam Mann, and Dan Segal.
\newblock {\em {Analytic pro-p groups}}, volume~61.
\newblock {Cambridge University Press}, 2003.

\bibitem{Gor86}
A.~V. Goryaga.
\newblock {Example of a finite extension of an FAC-group that is not an
  FAC-group}.
\newblock {\em {Sibirsk. Mat. Zh}}, 27(3):203--205, 1986.

\bibitem{HWZ12}
Emily Hamilton, Henry Wilton, and Pavel~A. Zalesskii.
\newblock {Separability of double cosets and conjugacy classes in 3-manifold
  groups}.
\newblock {\em {Journal of the London Mathematical Society}}, pages --040,
  2012.

\bibitem{hempel14}
John Hempel.
\newblock {Some 3-manifold groups with the same finite quotients}.
\newblock {\em {arXiv preprint arXiv:1409.3509}}, 2014.

\bibitem{Hig64}
Graham Higman.
\newblock {Amalgams of $p$-groups}.
\newblock {\em {Journal of Algebra}}, 1(3):301--305, 1964.

\bibitem{Iva04}
E.~A. Ivanova.
\newblock {On the conjugacy separability in the class of finite $ p $-groups of
  finitely generated nilpotent groups}.
\newblock {\em {arXiv preprint math/0408393}}, 2004.

\bibitem{kob09}
Thomas Koberda.
\newblock {Residual properties of 3-manifold groups {I}: Fibered and hyperbolic
  3-manifolds}.
\newblock {\em {arXiv preprint arXiv: 0910. 2035}}, 2009.

\bibitem{kob13}
Thomas Koberda.
\newblock {Residual properties of fibered and hyperbolic 3-manifolds}.
\newblock {\em {Topology and its Applications}}, 160(7):875--886, 2013.

\bibitem{Mar06}
Armando Martino.
\newblock {A proof that all Seifert 3-manifold groups and all virtual surface
  groups are conjugacy separable}.
\newblock {\em {Journal of Algebra}}, 313(2):773--781, 2007.

\bibitem{MM12}
Armando Martino and Ashot Minasyan.
\newblock {Conjugacy in normal subgroups of hyperbolic groups}.
\newblock In {\em {Forum Mathematicum}}, volume~24, pages 889--909, 2012.

\bibitem{Nib92}
G.~A. Niblo.
\newblock {Separability properties of free groups and surface groups}.
\newblock {\em {Journal of Pure and Applied Algebra}}, 78(1):77--84, 1992.

\bibitem{Par09}
Luis Paris.
\newblock {Residual $p$ properties of mapping class groups and surface groups}.
\newblock {\em {Transactions of the American Mathematical Society}},
  361(5):2487--2507, 2009.

\bibitem{PW12}
Piotr Przytycki and Daniel~T. Wise.
\newblock {Mixed 3-manifolds are virtually special}.
\newblock {\em {arXiv preprint arXiv:1205.6742}}, 2012.

\bibitem{RZup}
L.~Ribes and P.~Zalesskii.
\newblock {Profinite Trees}.
\newblock {\em {Unpublished book}}, 2001.

\bibitem{RZ00p}
Luis Ribes and Pavel Zalesskii.
\newblock {Pro-$p$ trees and applications}.
\newblock In {\em {New horizons in pro-$p$ groups}}, pages 75--119. {Springer},
  2000.

\bibitem{RZ00}
Luis Ribes and Pavel Zalesskii.
\newblock {\em {Profinite groups}}.
\newblock {Springer}, 2000.

\bibitem{RZ96}
Luis Ribes and Pavel~A. Zalesskii.
\newblock {Conjugacy separability of amalgamated free products of groups}.
\newblock {\em {Journal of Algebra}}, 179(3):751--774, 1996.

\bibitem{scott83}
Peter Scott.
\newblock {The geometries of 3-manifolds}.
\newblock {\em {Bulletin of the London Mathematical Society}}, 15(5):401--487,
  1983.

\bibitem{Ste70}
P.~F. Stebe.
\newblock {A residual property of certain groups}.
\newblock {\em {Proceedings of the American Mathematical Society}},
  26(1):37--42, 1970.

\bibitem{Ste72}
P.~F. Stebe.
\newblock {Conjugacy separability of certain Fuchsian groups}.
\newblock {\em {Transactions of the American Mathematical Society}},
  163:173--188, 1972.

\bibitem{thurstonnotes}
William~P. Thurston.
\newblock {The geometry and topology of three-manifolds}.
\newblock {URL: http://library.msri.org/books/gt3m/}, 1978.

\bibitem{Wilk15}
Gareth Wilkes.
\newblock {Profinite rigidity for Seifert fibre spaces}.
\newblock {\em {arXiv preprint arXiv:1512.05587}}, 2015.

\bibitem{Wilk16}
Gareth Wilkes.
\newblock {Profinite detection of JSJ decompositions of 3-manifolds}.
\newblock {\em {arXiv preprint arXiv:1605.08244}}, 2016.

\bibitem{WZ10}
Henry Wilton and Pavel Zalesskii.
\newblock {Profinite properties of graph manifolds}.
\newblock {\em {Geometriae Dedicata}}, 147(1):29--45, 2010.

\bibitem{WZ14}
Henry Wilton and Pavel Zalesskii.
\newblock {Distinguishing geometries using finite quotients}.
\newblock {\em {arXiv preprint arXiv:1411.5212}}, 2014.

\end{thebibliography}
\end{document}